\documentclass[11pt,regno]{amsart}

\usepackage{amscd,amssymb,amsmath,amsthm}
\usepackage[arrow,matrix]{xy}
\usepackage{graphicx}
\usepackage{color}
\usepackage{cite}
\usepackage{hyperref}
\usepackage{enumerate}
\usepackage{mathtools}
\usepackage{caption, subcaption}
\usepackage{tikz}
\usepackage{tikz-3dplot}
\usepackage{tikz-cd}

\usetikzlibrary{calc}
\usetikzlibrary{positioning}
\usetikzlibrary{decorations.markings}
\usetikzlibrary{arrows,snakes,backgrounds}

\tikzstyle{place}=[circle,draw=black,fill=blue!20,thick,
				inner sep=0pt,minimum size=6mm]				
\tikzstyle{dot}=[circle,draw=black,fill=black!20,thick,
				inner sep=0pt,minimum size=1mm]
\tikzstyle{rec}=[rectangle,draw=black,fill=black!20,thick,
				inner sep=0pt,minimum size=1mm]				
\tikzstyle{pre}=[<-,shorten <=0pt, >=stealth',semithick]
\tikzstyle{post}=[->,shorten <=0pt, >=stealth',semithick]

\newcounter{dummy} \numberwithin{dummy}{section}

\newtheorem{thm}[dummy]{Theorem}
\newtheorem{defn}[dummy]{Definition}
\newtheorem{lem}[dummy]{Lemma}

\newtheorem{pro}[dummy]{Proposition}

\newtheorem{rem}[dummy]{Remark}

\DeclareMathOperator{\br}{br}

\DeclarePairedDelimiter{\ceil}{\lceil}{\rceil}
\DeclarePairedDelimiter{\floor}{\lfloor}{\rfloor}

\def\a{\alpha}
\def\b{\beta}
\def\g{\gamma}

\def\d{\delta}
\def\D{\Delta}
\def\l{\lambda}
\def \L {\Lambda}

\def\s{\sigma}

\def\O{\Omega}

\def\o{\omega}

\def\L{\Lambda}

\def\Z{\mathbb{Z}}
\def\R{\mathbb{R}}
\def\N{\mathbb N}

\def\SS{\Subset}

\def\CC{\mathcal{C}}

\def\FF{\mathcal{F}}

\def\TT{\mathcal{T}}
\def\PP{\mathcal{P}}
\def\M{\mathcal M}

\pgfdeclarelayer{bg}    
\pgfsetlayers{bg,main} 

\numberwithin{equation}{section} 

\numberwithin{dummy}{section}

\begin{document}
\title[Non-robust phase transitions on trees]{Non-robust phase transitions in the generalized clock model on trees}

\author{C. K\"ulske, P. Schriever}

\address{C.\ K\"ulske\\ Fakult\"at f\"ur Mathematik,
Ruhr-University of Bochum, Postfach 102148,\,
44721, Bochum,
Germany}
\email {Christof.Kuelske@ruhr-uni-bochum.de}

\address{P.\ Schriever\\ Fakult\"at f\"ur Mathematik,
Ruhr-University of Bochum, Postfach 102148,\,
44721, Bochum,
Germany}
\email {Philipp.Schriever-d8j@ruhr-uni-bochum.de}

\begin{abstract} 
{\bf } 
Pemantle and Steif provided a sharp threshold for the existence of a RPT (robust phase transition) for the continuous 
rotator model and the Potts model 
in terms of the branching number and the second eigenvalue of the transfer operator, 
where a robust phase transition is said to occur if an arbitrarily weak coupling 
with symmetry-breaking boundary conditions suffices to induce symmetry breaking in the bulk. 
They further showed that for the Potts model RPT occurs at a different threshold than PT (phase transition in the sense of multiple Gibbs measures), 
and conjectured that RPT and PT should occur at the same threshold in the continuous rotator model. 

We  consider the class of $4$- and $5$-state rotation-invariant spin models with reflection symmetry on general trees which contains 
the Potts model and the clock model with scalarproduct-interaction  as limiting cases. 
The clock model can be viewed as a particular discretization which is obtained from the 
classical rotator model with state space $S^1$.  


We analyze the transition between PT=RPT and PT$\neq$RPT,
in terms of the eigenvalues of the transfer matrix of the model at the critical 
threshold value for the existence of RPT. The transition between the two regimes depends sensitively on the third largest eigenvalue.  

\end{abstract}

\maketitle

\thispagestyle{empty}


\section{Introduction}
We consider spin models on locally finite trees $T=(V,E)$ where the spin variables $\sigma_x, x \in V$, take values in the finite local state space $\O_0 = \{ 0,...,q-1\}$, $q \in \N$. 
The interaction between the spin variables will be given by a transfer matrix which is circulant, i.e. possesses discrete rotation-invariance, and which can therefore be uniquely described by its eigenvalues. 
An interesting question is whether or not there exists more than one extremal Gibbs measure, i.e., if there is a phase transition for this model. A little bit more involved is the question if it is possible to identify precise values for the free parameters of the transfer matrix at which a phase transition occurs. 
For background and recent results see \cite{BiEnEn16,BlGa90,EnErIaKu12,FoKu09,GaRuSh12,GaRo97,KuRo14,KuRo16,KuSc16,PP2010,Sly}.\\
Another interest in spin models on trees comes from network theory, where random graphs are used to describe the interaction between different agents. In these models the random graphs are often assumed (or proven in the large-graph limit) to be locally tree-like. To understand the behavior of these systems it is important to understand the situation on the tree first \cite{DeMo10, DeMoSu13,DoGiHo14,DoKuSc16}. 

For the Ising model on an arbitrary infinite tree with constant interaction strength $J$ the exact critical inverse temperature $\b_c$ of phase transition is well-known and is given by $J \b_c = k \coth^{-1}(\br T)$, where $k$ is Boltzmann's constant and $\br T$ is the {\em branching number} of the tree which captures the average number of edges per vertex \cite{Ly89}.
Similarly, the critical value for the existence of infinite clusters under independent Bernoulli percolation and the critical value for recurrence of random walks on the tree are given in terms of $\br T$ as well \cite{LyRW}.

One way to determine whether or not the spin model exhibits a phase transition is to exploit the recursion formula for the marginal distributions of the finite-volume Gibbs measures under boundary condition, see \eqref{rec}. 
For such models there is no phase transition, if and only if all the single-site marginals of the finite-volume Gibbs measures converge to the equidistribution for every boundary condition $\o \in \O$ (Proposition \ref{pro:rep}). For some models which are stochastically increasing in the boundary condition, in a sense described below, it is even enough to prove that the marginals of the finite-volume Gibbs measure under {\em plus boundary condition} converge to the equidistribution. This is for example the case in the Ising and the continuous rotator model, but also in the {\em generalized clock model} which will be introduced shortly (Proposition \ref{pro:plus}). 


Showing that the Gibbs marginals converge under the tree recursion is in general a challenging problem: Along the leaves of the finite subtree we have to fix a boundary condition, which means that we are starting from a Dirac distribution, i.e., far away from the equidistribution, and the recursion, in general, lacks any convexity.
 
In \cite{RPT} Pemantle and Steif introduced the notion of a {\em robust phase transition} (or RPT) which makes the tree recursion for the Gibbs marginals easier to analyze. Such a robust phase transition is said to occur when even for an arbitrary weakening of the interaction along the spins at the boundary of a sequence of finite sub-volumes $\L_n \SS V$, we still retain information about the spins at a single site in the thermodynamic limit $n \to \infty$, i.e., even under this stronger condition the single-site Gibbs marginals are not converging to the equidistribution. For more background on phase-transitions in systems 
with weak boundary couplings see also \cite{En00,EnMeNe02}.
The question whether or not there is a {\em robust} phase transition is easier to answer as one can make use of local arguments around the equidistribution. To be more specific: Under the assumption that there is a robust phase transition it is possible to start the recursion in a distribution which is already arbitrarily close to the equidistribution. Then, we only have to make sure that under the recursion we are repelled from the equidistribution. The behavior is heuristically understood in terms of the linearization of the recursion, but to control the non-linearities along relevant directions is delicate.    

Pemantle and Steif provided a sharp threshold for the existence of a RPT for the continuous rotator model (also called classical 
Heisenberg model) in any dimension, as well as for the Potts model. Similarly to the threshold for {\em regular} phase transitions in the Ising model, it is given in terms of the  branching number of the tree and an effective coupling parameter, which turns out to be the second eigenvalue of the transfer operator. 
Furthermore they showed that for the Potts model RPT occurs at a different threshold than the regular phase transition, but conjectured that RPT and regular phase transition should occur at the same threshold in the continuous Heisenberg models \cite{RPT}. 
It is therefore interesting to investigate generalized clock models as interpolations between the Potts model and the model of plane rotators 
(XY-model, Heisenberg model on $S^1$), from 
the perspective of the study of RPTs. 
(For more results around the issue of differences and similarities between continuous spin-models and clock models see e.g. \cite{FrSp81, MaSh11, EnKuOp11,JaKu14}.) 

This paper consists of two main parts:\\

\textit{Applicability of Pemantle-Steif theory}: \\
We consider a class of finite-state spin models on general trees and show that for a wide range of circulant transfer matrices the sharp threshold for the existence of RPT given by Pemantle and Steif holds in these cases as well (Proposition \ref{pro:rpt}). 
More precisely, this result holds under the assumption that the matrix elements of the transfer matrix are non-increasing in the distance to the diagonal.  
This is equivalent to saying that the pair potential describing the interaction of two neighboring spins along the tree is non-decreasing in the difference 
of the spin values along the discrete circle. 
\\
 
\textit{Bifurcation analysis on binary tree for $q=4,5$}: \\
For $q=4,5$ we show that in these models, which we call {\em generalized clock models}, RPT and regular phase transition are in general not equivalent and we determine the transition point between both scenarios in terms of the third eigenvalue of the corresponding transfer operator for two examples on the binary tree (Theorem \ref{thm:rpt45}). For a visualization of the regions of the different regimes see Figures \ref{PTvsRPT}, \ref{PTvsRPTq5} in Section 4. 
In the standard clock model with cosine-potential it turns out that there are no non-robust phase transitions (unlike $q=3$). 
\\
The non-equivalence is shown by studying the tree recursion on the level of the Fourier coefficients $\a_1, \a_2$ of the marginal distributions of the finite-volume Gibbs measures. This leads us to a two-dimensional fixed-point equation which turns out to be polynomial. The number of solutions to this fixed-point equation can then be given by analytic arguments.


\section{The generalized clock model: Definitions and main results}


Let $T=(V,E)$ be a locally finite tree with vertex set $V$ and edge set $E$ which is rooted at some vertex $o \in V$. The number of neighbors of the vertices is assumed to be bounded by a finite constant $B$ throughout. Let $d_G: V \times V \to\mathbb{N}$ be the graph distance on $T$, i.e., $d_G(x,y)$ equals the number of edges in the shortest path from $x$ to $y$ for each pair $x,y \in V$. Also, let $|x|:=d_G(x,o)$ for any $x\in V$. 
Two vertices $v$ and $w$ are called nearest neighbors if $d_G(v,w)=1$, i.e., if there exists an edge $e\in E$ connecting them. 
We use the notation $e=\{ v,w \}$. We write $v \to w$ if $d_G(v,w)=1$ and $|w|=|v|+1$, i.e., $v$ is the parent of $w$.  Finite subsets $\L$ of $V$ will be denoted by $\L \SS V$.

Every vertex $v\in V$ will be equipped with a spin variable $\sigma_v$ taking values in the finite metric space $(\O_0 , d)$, where $\O_0 :=\lbrace 0,...,q-1 \rbrace$ and $d(k,l):=1- \cos\big(2\pi(k-l)/q\big)$. Let $G:= \mathbb{Z}_q$ be the \textit{cyclic group} acting transitively on $\O_0$. 
Some $|\O_0| \times |\O_0|$ probability transition matrix $M$ is chosen, s.t. $$M(k,l) = M(g\circ k,g \circ l)  \quad \forall k,l \in \O_0, \;  \forall g\in G,$$
and $M(0,j) = M(0,k)$ if $d(0,j) = d(0,k)$.
We will further assume that $M$ is { \em non-increasing}, i.e., $M(i,j) \leq M(k,l)$ if $d(i,j) \geq d(k,l)$. 
Due to the Perron-Frobenius theorem the stochastic matrix $M$ has the distinguished eigenvalue $\l_0=1$ with the following properties: First, $|\l|<\l_0$ for all eigenvalues $\l\neq \l_0$ of $M$, and secondly, $\l_0$ is simple. We note that because $M$ is invariant under the $G$-action the matrix is circulant and can hence be completely described by its first row, which will be denoted by $r$.

As $M$ is a transfer matrix (which even happens to be a stochastic matrix the way we have 
defined it) we can define an associated potential $\Phi(i,j)$ by  letting 
\begin{equation*}\begin{split}
\Phi(i,j) :=-\log M(i,j)
\end{split}\end{equation*}
and define the corresponding formal infinite-volume Hamiltonian by 
\begin{equation*}\begin{split}
H(\o)=\sum_{\{ v,w \} \in E}\Phi(\o_v,\o_w).
\end{split}\end{equation*}
We call a spin model with such a Hamiltonian a \textit{generalized clock model}. 
Using inverse discrete Fourier transformation we can describe the circulant transfer matrix $M$ (and the potential $\Phi$ accordingly) also by its eigenvalues. Remember that the eigenvalues $\l_j$ of $M$ are given by
$$\l_j = \sum_{k=0}^{q-1} r_k e^{-2\pi i j k /q}, \quad j\in \{ 0,...,q-1\}$$
and the first row of $M$ is then recovered by
$$r_l = \frac{1}{q} \sum_{k=0}^{q-1} \l_k e^{2\pi ilk / q}, \quad l\in \{0,...,q-1\}.$$ \\
If the eigenvalues are chosen such that $M$ is of the form 
$$M_{i,j} = \left\{\begin{array}{cl} e^{J} /Z, & \mbox{if } i= j \\ 1 /Z ,& \mbox{if }i\neq j \end{array}\right.$$
where $J>0$ suitable and $Z>0$ is a normalizing constant, we recover the \textit{Potts model}. If $M$ is given by
$$M_{i,j} = \frac{1}{Z} e^{J\cos(2\pi (i-j)/q)},$$
for some $J>0$ and normalizing constant $Z>0$, we recover the \textit{standard-clock model} with scalarproduct-interaction. 

A \textit{cutset} $C$ is a finite set of vertices such that every self-avoiding infinite path starting from the root $o$ intersects $C$ such that there are no vertices $x,y \in C$ with $|x|<|y|$ and $x$ lies in the shortest path from $o$ to $y$. For a cutset $C$ the graph $T\setminus C$ consists of some infinite components and one finite component. The finite component  will be denoted by $C^i$ and the union of the infinite components of $T\setminus C$ will be denoted by $C^o$. A sequence of cutsets $\{ C_n \}_{n\in\mathbb{N}}$ is said to {\em exhaust} V, if for every vertex $v\in V$ there exists a $N\in\mathbb{N}$ s.t. $v \in C_n^i$ for every $n\geq N$. We define the configuration space by $\O = \O_0^V$. 

\begin{defn}
Let $C$ be a cutset of the tree $T$ and $\d \in \O$ a spin configuration 
which serves as a boundary condition to the interior.
The {\em finite-volume Gibbs measure on $\O_0^{C^i}$ under boundary condition $\d$} is defined by
\begin{equation}\label{gb}
\mu_C^{\Phi, \d}(
\o_{C^i}) = \frac{1}{Z_C^{\Phi}(\d)} \exp\Big( - H_C^{\Phi, \d} (\o_{C^i})\Big),
\end{equation}
where
$$H_C^{\Phi, \d}(\o_{C^i}) = \sum_{\substack{\{ x,y\} \in E: \\ x,y \in C^i}} \Phi(\o_x,\o_y) + \sum_{\substack{\{ x,y\} \in E: \\ x\in C^i ,y \in C}} \Phi(\o_x,\d_y)$$
and $$Z_C^{\Phi}(\d) = \sum_{\tilde \o_{C^i}\in \O_0^{C^i}} \exp \Big( - H^{\Phi, \d}_C(\tilde \o_{C^i}) \Big)$$ is a normalizing constant or {\em partition function}. If the second summand in the Hamiltonian $H$ is left out, i.e., if there is no interaction with the boundary, we call this the {\em free Gibbs measure}. 
\end{defn}

Note that the family of finite-volume Gibbs measures as we have defined them above can also be interpreted as a {\em (local) Gibbsian specification} $\g^\Phi$ \cite[Definition 1.23]{Ge88} which is given by
\begin{equation}\label{GibbsSpec}
\g^\Phi_{C^i}(f \mid \d) = \int \mu_C^{\Phi,\d}(d\o_{C^i}) f(\o_{C^i} \d_{(C^i)^c}). 
\end{equation}
This specification is a family of probability kernels with the property that $\g^\Phi_{C^i}(\cdot \mid \d) \in \M_1(\O_0^V)$ for every $\d \in \O$ and s.t. $\g^\Phi_{C^i}(A \mid \cdot)$ is measurable w.r.t. the sigma-algebra generated by the spin variables $\{ \s_v, v\in(C^i)^c \}$ for every $A\in (\PP(\Z_q))^V$.

Let a measure space $(X, \FF_X, \mu)$, a measurable space $(Y, \FF_Y)$ and a probability kernel $\pi$ from $\FF_X$ to $\FF_Y$ be given. Then the {\em push-forward} of $\mu$ under $\pi$ is defined by 
\begin{equation*}
\mu \pi(A) = \int \mu(d\o) \pi(A \mid \o) \quad \forall A \in \FF_Y. 
\end{equation*}
A measure $\mu \in \M_1(\O_0^V)$ is called a {\em Gibbs measure} w.r.t. a specification $(\g_\L)_{\L \SS V}$ if 
\begin{equation*}
\mu = \mu \g_\L \quad \forall \L \SS V.
\end{equation*} 
The probability kernels given by $\eqref{GibbsSpec}$ are only defined for finite sub-volumes $\L \SS V$ of the specific form $\L = C^i$ where $C$ is a cutset, but this is not an issue w.r.t. identifying the set of Gibbs measures:

\begin{lem}
Suppose $\g$ is a specification and $\mu$ a measure on $\O_0^V$. Then the following statements are equivalent:
\begin{enumerate}[(i)]
\item
$\mu$ is a Gibbs measure.

\item
$\mu \g_\L = \mu$ for every $\L \SS V$.

\item
$\mu \g_{C^i} = \mu$ for every cutset $C \SS V$.
\end{enumerate}
\end{lem}

For a proof, see \cite[Remark 1.24]{Ge88}. This lemma allows us to give an alternative characterization of Gibbs measures (in the infinite volume) as follows:

\begin{lem}
A probability measure $\mu$ in the infinite volume $\O_0^{V}$ is called a {\em Gibbs measure} for the potential $\Phi$ if for any cutset $C$ the conditional distribution on $C^i$ given a boundary configuration $\d \in \O_0^V$ is given by $\mu_C^{\Phi, \d}$, i.e.,
$$\mu(\o_{C^i} \mid \d_{C\cup C^o}) = \mu_C^{\Phi, \d}(\o_{C^i}).$$   
\end{lem}

Given an interaction potential $\Phi$, a cutset $C$ and a boundary condition $\d$, let $p_{C,o}^{\Phi,\d}$ denote the marginal distribution of the finite-volume Gibbs measure $\mu_C^{\Phi,\d}$ at the root $o.$ We now come to the key definition of a {\em robust phase transition}, which is said to occur if a ``plus'' boundary condition retains its influence on the root even when the interaction along the edges at the boundary is made arbitrarily small:
Given $\Phi$ and $u \in (0,1]$ and a cutset $C$ of $T$, let $\bar\Phi(u,C)$ be the potential which is $\Phi$ on edges in $C^i$ and $u \Phi$ 
on edges connecting $C$ to $C^i$. Let $\d_0$ be a boundary configuration which is all $0$, i.e., $(\d_0)_v = 0$ for all $v \in V$. 
Let $p_{C,o}^{u,\Phi,+}$ denote the marginal distribution at the root $o$ of the measure $\mu_C^{u,\Phi,+} := \mu_C^{\bar\Phi(u,C),\d_0}$.
In the following we will always denote the equidistribution on $\O_0$ by $\mathbf{1}$.

\begin{figure}
	\beginpgfgraphicnamed{pastedge}
		\begin{tikzpicture}[scale=0.8]
		
		\clip (-8.7,-1) rectangle (5,9);
		
		\node[dot] (o) at (2,0) {};
		
		\node[dot] (v) at (-2,2) {}
			edge		(o);
		\node[dot] (v1) at (6,2) {}
			edge 	(o);
			
		\node[dot] (w1) at (-6,4) {}
			edge 	(v);
		\node[dot] (w2) at (-2,4) {}
			edge 	(v);
		\node[dot] (w3) at (2,4) {}
			edge 	(v);
			
		\node[dot] (w11) at (-7,6) {}
			edge[dashed]  	(w1);
		\node[dot] (w12) at (-5,6) {}
			edge[dashed]  	(w1);
		
		\node[dot] (w21) at (-2,6) {}
			edge[dashed]  	(w2);	
		
		\node[dot] (w31) at (0,6) {}
			edge[dashed] 	(w3);	
		\node[dot] (w32) at (2,6) {}
			edge[dashed] 	(w3);	
		\node[dot] (w33) at (4,6) {}
			edge[dashed] 	(w3);	
			
	\draw[dotted] (-8,6) -- (5,6);
			
		\node[dot] (w111) at (-7.5,7) {}
			edge  	(w11);
		\node[dot] (w112) at (-6.5,7) {}
			edge 	(w11);
		
		\node[dot] (w121) at (-5.5,7) {}
			edge  	(w12);
		\node[dot] (w122) at (-4.5,7) {}
			edge 	(w12);	
			
		\node[dot] (w211) at (-2.5,7) {}
			edge 	(w21);	
		\node[dot] (w212) at (-2,7) {}
			edge 	(w21);	
		\node[dot] (w213) at (-1.5,7) {}
			edge 	(w21);	
		
		\node[dot] (w311) at (0,7) {}
			edge 	(w31);
			
		\node[dot] (w321) at (1.5,7) {}
			edge 	(w32);
		\node[dot] (w322) at (2.5,7) {}
			edge 	(w32);
			
		\node[dot] (w331) at (3.5,7) {}
			edge 	(w33);
		\node[dot] (w332) at (4.5,7) {}
			edge 	(w33);

		\node (delta) at (-8.5,6) {$\d$};
		\node (C) at (1,6.3) {$C$};
		\node (wurzel) at (2,-0.5) {$o$};
		\node (vau) at (-2,1.5) {$v$};
		
	\draw[dashed] (-8.6,0)--(-7.6,0);
	\node (dashed) at (-7,0) { $u \Phi$};
	
	\draw (-8.6,-0.5)--(-7.6,-0.5);
	\node (full) at (-7,-0.5) { $\Phi$};
		

%
%
%
%
%
%
%
%
%
%
%
%
			
		
		\end{tikzpicture}
	\endpgfgraphicnamed

	\label{fig:pastedge}
\end{figure}

\begin{defn}\label{rpt}
A generalized clock model on the tree $T$ is said to exhibit a {\em robust phase transition} for the interaction $\Phi$ if for every $u\in (0,1],$
$$\inf_C || p_{C,o}^{u,\Phi,+} - \mathbf{1} ||_{\infty} \neq 0,$$
where the infimum is taken over all cutsets $C$ of $T$.
\end{defn}


We will show that under a positivity assumption on the eigenvalues of the transfer matrix $M$, which is chosen such that $M$ is non-increasing, the existence of a robust phase transition is a geometric property of the underlying tree $T$. The decisive parameter in this regard is the {\em branching number}, which in some sense captures the average number of edges per vertex of the tree. It is defined by
\begin{equation}\label{br}
\br T := \sup \{ \l : \inf_C \sum_{v\in C} \l^{-|v|}>0 \},
\end{equation}
where the infimum is once more taken over all cutsets $C$. 

The branching number has been introduced  in this form by Lyons \cite{LyRW} and it can be interpreted as the exponential of the Hausdorff dimension of the boundary $\partial T$ of the tree which has been studied previously by Furstenberg \cite{Fu}. For percolation and symmetric root-biased random walks on trees it is known that there exist sharp thresholds for the existence of infinite clusters and positive recurrence respectively that are given by the branching number \cite{LyRW}. 

The definition of robust phase transitions goes back to a paper by \mbox{Pemantle} and Steif \cite{RPT}, in which they showed that under certain ``ferromagneticity'' assumptions, amongst others regarding the positivity of the Fourier coefficients of the Gibbs marginals, there exists a sharp threshold for the existence of such RPTs which is given by
$$\l_1 \br T  =1,$$
where $\l_1$ denotes the first non-trivial Fourier coefficient of the transfer operator. 
By establishing applicability of the results of Pemantle and Steif we prove that this threshold holds also for certain types of generalized clock models:

\begin{pro}\label{pro:rpt} 
Let $T$ be a tree and $M$ the transfer matrix of a generalized clock model with state space $\O_0 = \{0,1,...,q-1\}$. Assume that all eigenvalues of $M$ are non-negative and such that $M$ is non-increasing. Let $\l_1$ be the second largest eigenvalue of $M$. If
$$\l_1 \br T   > 1$$
then there is a robust phase transition and if
$$\l_1 \br T  <1$$
there is no robust phase transition. 
\end{pro}

Remember that for suitably chosen eigenvalues for the transfer matrix $M$ we recover the Potts model and the standard-clock model with the potential $\Phi(i,j) = \cos\big( \frac{2\pi}{q}(i-j)\big)$ respectively from the generalized clock model. In both of these two models the same sharp threshold for the existence of a RPT as given in Theorem \ref{pro:rpt} holds. 

The existence of a robust phase transition always implies phase transition (Proposition \ref{pro:plus}). 
However for the Potts model it is known that the existence of regular phase transitions does not solely depend on the second largest eigenvalue and the existence of a regular phase transition does not imply the existence of a RPT. For the classical Heisenberg model it has been only conjectured that robust and non-robust phase transitions do coincide \cite{RPT}. 

We show that in the generalized clock model there is no equivalence of PT and RPT in general: For $q=4,5$ on the binary tree we provide the exact transition point between the two regimes PT$=$RPT and PT$\neq$RPT at criticality, i.e., for $\l_1 = 1/2$, in terms of the third largest eigenvalue $\l_2$. For $q=4$ we give the transition points for every $\l_1 \in (0, 1/2)$ (see Figure \ref{PTvsRPT}).

\begin{thm}\label{thm:rpt45}
Consider the generalized $q$-state clock model on the binary rooted tree. For $q=4$ and $q=5$ there exists a non-empty region for the eigenvalues $\l_1, \l_2$ of the transfer matrix $M$, such that the spin system exhibits a phase transition but no robust phase transition. 
\end{thm}

\section{RPT is a geometric property: Applicability of Pemantle-Steif theory for the generalized clock model on possibly irregular trees}
An important property of Gibbs measures for spin models on trees is the recursive nature of their marginal distributions $p_{C,v}^{\Phi,\d}$: Let $w_1,...,w_k$ be the children of $v$. Then
\begin{equation}\label{rec}
p_{C,v}^{\Phi,\d}(i) = \frac{1}{Z_{C,v}^{\Phi,\d}} \prod_{l=1}^k \left( \sum_{j=0}^{q-1} M(i,j) p_{C,w_l}^{\Phi,\d}(j) \right),
\end{equation}
where
\begin{equation*}
Z_{C,v}^{\Phi,\d} = \sum_{k=0}^{q-1} \prod_{l=1}^k \left( \sum_{j=0}^{q-1} M(k,j) p_{C,w}^{\Phi,\d}(j) \right)
\end{equation*}
is a normalizing constant \cite[Lemma 2.2]{RPT}. We like to note that equation \eqref{rec} is the defining property of so-called {\em boundary laws} or {\em entrance laws} \cite{Ge88}, \cite{Z83}. 
It is known that every Gibbs measure which is a tree-indexed Markov chain can be uniquely represented in terms of a boundary law (up to a positive pre-factor), and conversely that every boundary law can be used to construct a Gibbs measure which is a Markov chain \cite{Z83}. Furthermore, the extremal elements of the set of Gibbs measures are necessarily Markov chains (see \cite[Theorem 12.6]{Ge88}), and therefore spin models on trees exhibit a phase transition if and only if there exists more than one boundary law.
Note that we always have the free solution, i.e., the equidistribution $\mathbf{1}$. 
For a recent paper showing the connection between gradient Gibbs measures and periodic boundary laws for models with non-compact local state space, see \cite{KuSc16}.

In the following we will work towards an equivalent definition of phase transition.
For every neighboring pair of vertices $v, w \in V$ we get two infinite trees when we remove the edge $\{ v, w \}$ from $E$. Let us denote the tree containing $v$ by $T_v$ and the other by $T_w$. 
For any cutset $C_n$ on the whole tree $T$ we define cutsets $C_n^v:= C_n \cap T_v$ and $C_n^w:= C_n \cap T_w$ on the two subtrees $T_v, T_w$. 
For any configuration $\d \in \O$ we introduce boundary conditions $\d_v$ and $\d_w$ on $T_v$ and $T_w$ 
by setting $\d_v := \d \cap T_v$ and $\d_w := \d \cap T_w$ respectively. 
Note that in the following we denote the marginal distributions of the finite-volume Gibbs measures on the infinite subtrees $T_v$ and $T_w$ by $p_{C_n^v, v}^{\Phi,\d_v}$ and  $p_{C_n^w, w}^{\Phi, \d_w} $ respectively. These are not to be confused with the marginals of the finite-volume Gibbs measure on the whole tree $T$. 
Let the {\em inner boundary} of a subset $\L \subset V$ be defined by $\partial^i \L := \{ v \in \L: \exists w \in \partial \L \mbox{ with } \{v,w\} \in E \}$. 

\begin{pro}\label{pro:rep}
Let $(\a_i)_{i\in V}$ be a family of measures on $\O_0$ and let $\{ C_n\}_{n\in\N}$ be a sequence of cutsets that exhausts $V$. Suppose that for every $v\in V$ and every $\d \in \O$ the marginal distribution $p_{C_n^{v}, v}^{\Phi, \d_v}$ 
converges to the distribution $\a_v$. 
Then the thermodynamic limit of the finite-volume Gibbs measure, $\mu := \lim_{n\to\infty} \g_{C_n^i}( \cdot \mid \d)$, exists for any choice of $\d$ and is given by
\begin{equation*}
\mu (\sigma_\L = \o_\L) = \frac{1}{Z_\L} \prod_{k \in \partial^i \L} \a_k(\o_k) \prod_{b \subset \L} M_b(\o_b)
\end{equation*} 
where $\L \SS V$ is any finite connected set, $\o_\L \in \O_0^\L$, and $Z_\L$ is a normalizing constant.
\end{pro}

\begin{proof}
The proof is done by induction on $|\L|$. 
First, let $\L = \{ v, w \}$ with $v \sim w$. Remove the edge connecting $v$ to $w$ from the tree to get the two subtrees $T_v$ and $T_w$ which contain $v$ and $w$ respectively. Note that by assumption
$$p_{C_n^v, v}^{\Phi, \d_v} (\o_v) = 
\frac{\sum_{\o_{C_n^{v,i}\setminus v}} \exp(-H_{C_n^{v,i}} (\o_v, \o_{C_n^{v,i}\setminus v}, \d^v_{(C_n^{v,i})^c} ) ) }{\sum_{\o_v, \o_{C_n^{v,i}\setminus v}} \exp(-H_{C_n^{v,i}} (\o_v, \o_{C_n^{v,i}\setminus v}, \d^v_{(C_n^{v,i})^c}) )}
\to \a_v(\o_v)$$
for $n \to \infty$. 
Let $g$ be some test-function on $\L = \{ v,w \}$. Then

\begin{equation*}\begin{split}
& \g_{C_n} (g(\o_v, \o_w) \mid \d^v, \d^w) \\
= & \sum_{\o_v, \o_w} \frac{1}{Z_{C_n}^\d} g(\o_v, \o_w) M_{vw}(\o_v, \o_w)   \\
& \qquad \times \left(\sum_{\o_{C_n^{v,i} \setminus \{v\} }} \exp \left(-H_{C_n^{v,i}} \left(\o_v, \o_{C_n^{v,i}\setminus v}, \d^v_{(C_n^{v,i})^c} \right) \right) \right) \\
& \qquad \qquad \times  \left(\sum_{\o_{C_n^{w,i} \setminus \{w\} }} \exp \left(-H_{C_n^{w,i}} \left(\o_w, \o_{C_n^{w,i}\setminus w}, \d^w_{(C_n^{w,i})^c} \right) \right) \right) \\
= & \frac{1}{Z_{v,w}} \sum_{\o_v, \o_w}  g(\o_v, \o_w) M_{vw}(\o_v, \o_w)  \, p_{C_n^v, v}^{\Phi, \d_v} \,  p_{C_n^w, w}^{\Phi, \d_w}      \\
\to  & \frac{1}{Z_{v,w}}  \sum_{\o_v, \o_w}  g(\o_v, \o_w) M_{vw}(\o_v, \o_w) \, \a_v(\o_v) \, \a_w(\o_w)
\end{split}\end{equation*}
for $n \to \infty$. 
This proves the claim for volumes $\L$ that consist of only two sites. 

Now suppose the claim holds for any sub-volume $\D \subset V$ with $|\D| = n$. Let $\L := \D \cup \{ w \}$ with $w \in \partial \D$ and $v:= \partial w \cap \D$.  Let $g$ be some test-function on $\L$. Then
\begin{equation*}\begin{split}
& \g_{C_n} (g (\o_\L) \mid \d^v, \d^w ) = \g_{C_n} (g (\o_\D, \o_j) \mid \d^v, \d^w ) \\
= & \sum_{\o_\D, \o_j} \frac{1}{Z_{C_n}^\d} g(\o_\D, \o_w) M_{vw} (\o_v, \o_w)  \\
& \qquad \times \left(  \sum_{\o_{C_n^{v,i}\setminus\D}}  
 \left(\exp\left(  -H_{C_n^{v,i}} (\o_\D, \o_{C_n^{v,i}\setminus\D}), \d_{(C_n^{v,i})^c} \right)\right) \right)\\
& \qquad   \qquad \times \left(  \sum_{\o_{C_n^{w,i}\setminus \{w\} }}      \left(\exp\left(  -H_{C_n^{v,i}} (\o_w, \o_{C_n^{v,i}\setminus w}), \d_{(C_n^{w,i})^c} \right)\right) \right)\\
= & \frac{1}{Z_{C_n}^\d}  \sum_{\o_\D, \o_j} g(\o_\D, \o_j)   M_{vw} (\o_v, \o_w)  \,   \g_{C_n^{v,i}}(\o_\D \mid \d_v) \, p_{C_n^w,w}^{\Phi, \d_w} \\
\to  & \frac{1}{Z_\L}   \sum_{\o_\D, \o_j}   g(\o_\D, \o_j)    M_{vw} (\o_v, \o_w)  \,   \left( \prod_{k \in \partial^i \D \setminus \{ w \}} \a_k(\o_k) \prod_{b \subset \D} M_b(\o_b) \right)   \a_w(\o_w).
\end{split}\end{equation*}
The last equation follows from the induction assumption. Note that $\g_{C_n^{v,i}} (\cdot \mid \d_v)$ is the measure on the subtree $T_v$ and not on the whole tree. Hence the inner boundary of $\D$ in $T_v$ is not containing $w$. 
This completes the proof.  
\end{proof}

Note that Proposition \ref{pro:rep} readily implies that there is no phase transition under the given assumptions: Every extremal Gibbs measure $\mu$ can be given as a thermodynamic limit of a finite-volume Gibbs measure prepared with some boundary configuration which is typical for $\mu$ (see \cite[Proposition 2.23]{EnFeSo93}). As these limits are assumed to be the same for {\em any} boundary configuration $\o$, we find that there can only exist one extremal Gibbs measure. 
Therefore the spin model exhibits a phase transition if and only if the recursion equation \eqref{rec} has more than one solution. It is even enough to study the recursion under ``plus'' boundary condition:

\begin{pro}\label{pro:plus}
In the generalized clock model a phase transition occurs if and only if there exists a vertex $v$ and a sequence of cutsets $\{ C_n \}_{n\in\N}$ such that
\begin{equation*}
\inf_n \Vert p_{C_n,v}^{\Phi, \d_0} - \mathbf{1} \Vert_\infty \neq 0. 
\end{equation*}
\end{pro}

Proposition \ref{pro:plus} also holds for the Ising, Potts and continuous rotator models. The proof for the generalized clock model is analogous to the one for the rotator model \cite[Proposition 1.4]{RPT}. The crucial property in these models is the positivity of the Fourier coefficients of the marginal distributions $p_{C_n, v}^{\Phi, \d_0}$ \eqref{eq:rpositivity} and of the transfer operator $r$ \eqref{eq:foupos} which we will discus later. \\

Observe that the r.h.s. of \eqref{rec} is a convolution of the different distributions of the children over the vectors in the circular matrix representing the interaction, 
followed by an ordinary product. We will make use of consequences of this structure regarding monotonicity below. 

In the following we will show that in the generalized clock model, where we assume that the eigenvalues of the transfer matrix $M$ are restricted by 
$$1\geq \l_j > 0, \quad \forall j\in \{0,...,q-1\},$$
and such that $M$ is non-increasing,
the existence of a RPT is determined by the branching number of the underlying tree and the second largest eigenvalue of the transfer matrix. 

To understand the recursion \eqref{rec} for the Gibbs marginals, it is important to note that it preserves the natural class of symmetric probability vectors 
which have positive Fourier (cosine-)coefficients, as we will explain now. \\
By $P(\O_0/0)$ we denote the vectors on $\O_0$ that are symmetric to $0$, i.e., $f(j)=f(k)$ if $d(0,j) = d(0,k)$ for any $f \in P(\O_0/0)$. Note that $\dim (P(\O_0/0))=\lfloor \frac{q}{2} \rfloor+1$. It will be convenient to work with the non-normalized basis $\mathcal{B}=\lbrace \phi_j \rbrace_{j=0,...,\lfloor\frac{q}{2} \rfloor}$, where 
$$\phi_j(k) := \cos (2\pi jk/q), \quad k \in \{0,1,...,q-1 \}.$$ 
This basis has the properties 
\begin{equation}\label{q}
\phi_i \phi_j = \frac{1}{2} \phi_{i+j} + \frac{1}{2} \phi_{i-j}
\end{equation} 
and
\begin{equation}\label{star}
\phi_i * \phi_j = z_j \d_{ij} \phi_j, 
\end{equation}
where $\phi_i *\phi_j(k) := \sum_l \phi_i(k-l)\phi_j(l)$ and $z_j = \sum_{k\in E} \phi_j^2(k)$. We introduce a new norm $||\cdot||_A$ on the space of symmetric vectors $P(\O_0/0)$, which is defined as the sum of the absolute values of the Fourier coefficients, i.e., 
$$||f||_A:= \sum_{j=0}^{\lfloor \frac{q}{2} \rfloor} |a_j(f)|,$$ 
where $a_j(f):= z_j^{-1}\langle f,\phi_j\rangle$. 

For any two distributions $f,g \in P(\O_0/0)$, we clearly have $fg \in P(\O_0/0)$ and 
$$||fg||_A \leq ||f||_A ||g||_A.$$
This submultiplicativity can be seen by developing $f$ and $g$ w.r.t. $\mathcal{B}$ and using property \eqref{q}.
As any $f\in P(\O_0/0)$ can be developed w.r.t. $\mathcal{B}$, property \eqref{star} gives us 
$$a_j(f*g) = z_j a_j(f)a_j(g)  \quad \forall f,g \in P(\O_0/0).$$
Note that 
\begin{equation}\label{eq:rpositivity}
a_j(Me_1) = a_j (r) = \l_j / z_j >0
\end{equation}
for all $j\geq 0$. 

We define the class $P_+(\O_0/0) \subset P(\O_0/0)$ as the smallest class of symmetric probability vectors which contains the canonical basis vector $e_1 = (1,0,...,0)$ and which is closed under the recursion given by \eqref{rec}. Note that this class contains all possible marginal distributions of any finite-volume Gibbs measure under ``plus'' boundary condition. By the recursion formula it follows, that all the Fourier coefficients of any marginal $g\in P_+(\O_0/0)$ of a Gibbs measure are strictly positive:
\begin{equation}\label{eq:foupos}
a_j(g) >0 \quad \forall j\in \{0,...,q-1\}.
\end{equation}
Remember that this is only the case since we have chosen the eigenvalues of the transfer matrix $M$ to be positive. 
Another important property is that the Fourier coefficients of higher order are dominated by the first one: Let $f\in P_+(\O_0/0)$. Then
$$\left\lvert 1 + 2\sum_{k=0}^{\floor*{\frac{q}{2}}} f(k) \cos \left( \frac{2\pi}{q}jk \right) \right\rvert 
\leq  1 + 2\sum_{k=0}^{\floor*{\frac{q}{2}}} f(k) \cos \left( \frac{2\pi}{q}k \right), \quad \forall j \in \{0,...,\floor*{q/2}\},$$
which is equivalent to say that $z_1a_1(f) \geq |z_j a_j(f)|$ for all $j \in \{0,...,\floor*{q/2}\}$. For $r$ being the first row of the transfer matrix $M$ of a generalized clock model with $M$ being non-increasing this implies that $\l_1$ is the second largest eigenvalue:
$$\l_1(M) = z_1a_1(f) \geq |z_j a_j(f)| = |\l_j|, \quad \forall j \in \{0,...,\floor*{q/2}\}.$$
We come to an important lemma which establishes an upper bound for the distance of the convoluted distributions to their linearization. 

\begin{lem}\label{lin}
There exists a function $o$ with $\lim_{h\rightarrow0}(o(h)/h)=0$ such that for all $h_1,...,h_k\in P_+(\O_0/0)$,
 \begin{equation*}
 \lVert \frac{1}{Z} \prod_{i=1}^k  h_i   -\mathbf{1}-\sum_{i=1}^k (h_i-\mathbf{1}) \rVert_A \leq o(\max_i ||h_i-\mathbf{1}||_A),
 \end{equation*}
 provided $\max_i ||h_i-\mathbf{1}||_A\leq1$. Here the product is understood to be pointwise and $Z>0$ is a normalizing constant turning $\prod_{i=1}^k  h_i $ into a probability vector. 
\end{lem}

This lemma is similar to Lemma 2.7 from \cite{RPT} and the proof works similarly. We like to point out that the proof requires a bound on the term $\sum_{j=0}^{q-1} \prod_{i=1}^k h_i(j)$ from below. This is where the non-negativity of the eigenvalues $\l_i$ of the transition matrix enters. We have seen that under this condition all Fourier coefficients of the vectors $h_i$ are positive. This and property \eqref{q} imply that $\sum_{j=0}^{q-1} \prod_{i=1}^k h_i(j) \geq \sum_{j=0}^{q-1} \prod_{i=1}^k a_0(h_i)\phi_0(j) = 1/q^{k-1}$. 

\begin{rem}\label{rem:lin}
Consider the $q$-state generalized clock model with $q=4$ or $q=5$ on the binary tree. Assume that $M$ is a non-increasing circulant matrix which may even have negative eigenvalues.
Let $h_1$ and $h_2$ be any elements of $P_+(\O_0/0)$ with $h_i = \frac{1}{q} \phi_0 + a_1(h_i) \phi_1 + a_2(h_i) \phi_2 $. 
For $q=4$ we get by the orthogonality of the basis functions that
$$\sum_{j=0}^{q-1} h_1(j) h_2(j) = \frac{1}{q} + \frac{q}{2} a_1(h_1)a_1(h_2) + q \; a_2(h_1)a_2(h_2).$$
For $q=5$ we get 
$$\sum_{j=0}^{q-1} h_1(j) h_2(j) = \frac{1}{q} + \frac{q}{2} \big(a_1(h_1)a_1(h_2) + a_2(h_1) a_2(h_2)\big).$$
For probability vectors $h \in P_+(\O_0)$ there are of course restrictions on the values of the Fourier modes $a_{1,2}$ that come from $0 \leq h(j) \leq 1$ for all $j \in \O_0$ and from the fact that $h$ is non-increasing.  
For the region of values the coefficients $a_{1,2}$ can attain in this case, see Figure \ref{fig:Lemma2.7}.
It can be easily checked that $\sum_{j=0}^{q-1} h_1(j) h_2(j)$ is always bounded from below for any choice of $h_1, h_2$. Therefore Lemma \ref{lin} holds for the generalized clock model with $q=4,5$ even for transfer matrices with non-positive eigenvalues. 
\end{rem}

\begin{figure}
\centering
\begin{subfigure}{.5\textwidth}
  \centering
  \includegraphics[width=0.9 \linewidth]{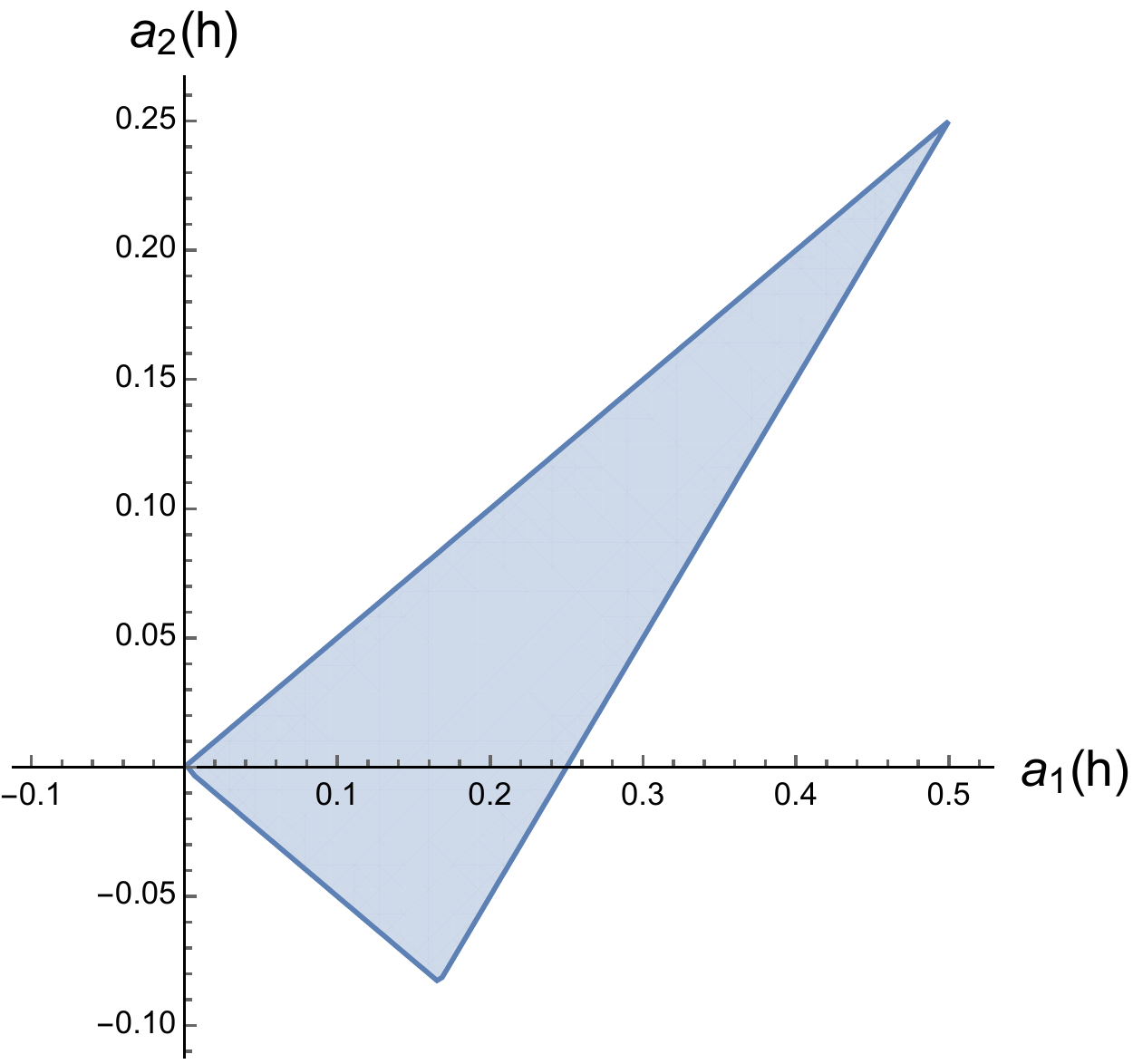}
  \caption{$q=4$.}
  \label{fig:q4Lemma2.7}
\end{subfigure}%
\begin{subfigure}{.5\textwidth}
  \centering
  \includegraphics[width=0.9 \linewidth]{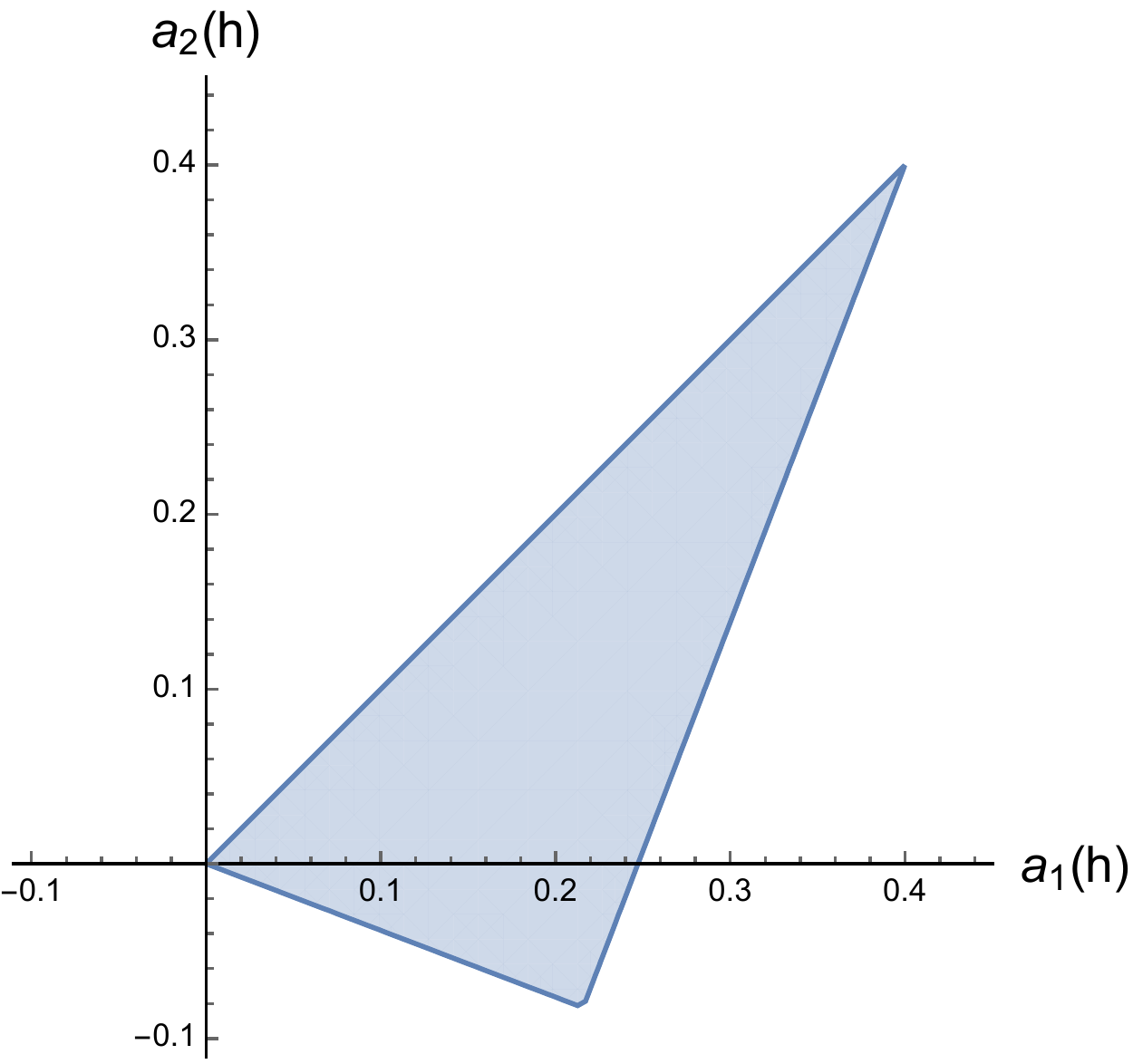}
  \caption{$q=5$.}
  \label{fig:q5Lemma2.7}
\end{subfigure}
\caption{Region of the coeffcients $a_1(h),a_2(h)$ for a probability vector $h\in P_+(\O_0/0)$, where $M$ is allowed to have non-positive eigenvalues.}
\label{fig:Lemma2.7}
\end{figure}

%
%
%
%

We come to the proof of Proposition \ref{pro:rpt} regarding robust phase transitions on trees in the generalized clock model. \\


\noindent\textit{Proof of Proposition \ref{pro:rpt}.} We will check the necessary conditions for Theorems 3.1 and 3.2 from the paper by Pemantle and Steif \cite{RPT}:\\

\begin{enumerate}[(i)]
\item From Lemma \ref{lin} follows the upper bound for the distance of the convoluted distributions to their linearization. 
\item Note that as $P(\O_0 /0)$ is finite-dimensional all norms are equivalent, in particular the $A$-Norm $||\cdot||_A$ and the supremum norm $||\cdot||_\infty$. 
\item Let $\Phi$ be a potential that corresponds to some circulant transfer matrix $M$. For $u\in (0,1]$ we define $M^u$ to be the transfer matrix that corresponds to the potential $u\Phi$, i.e., the first row of $M^u$ is given by 
$$r^u(j) = \exp\big(-u\Phi(0,j)\big).$$ 
Clearly $||r^u - \mathbf{1}||_{\infty}\to 0$ and hence $||r^u - \mathbf{1}||_A\to 0$ for $u\to0$. 

\item The transfer matrix $M$ is bounded by $1$ in the operator norm:
\begin{equation*}\begin{split}
||M||_A &= \sup_{||f||_A=1} ||Mf||_A = \sup_{||f||_A=1} ||r*f||_A\\
 &= \sup_{||f||_A=1} \sum_{j = 0}^{\lfloor q/2 \rfloor} |a_j(r*f)| 
= \sup_{||f||_A=1} \sum_{j = 0}^{\lfloor q/2 \rfloor} |\g_j a_j(r)a_j(f)|\\ 
& = \sup_{||f||_A=1} \sum_{j = 0}^{\lfloor q/2 \rfloor} |\l_j a_j(f)| \leq \sup_{||f||_A=1} \sum_{j = 0}^{\lfloor q/2 \rfloor} |a_j(f)| = 1.   
\end{split}\end{equation*}

\item The next inequality shows that the convolution is in some sense a contraction w.r.t. to the $A$-norm:
\begin{equation*}\begin{split}
||Mf- \mathbf{1}||_A &= ||r*f-\mathbf{1}||_A = \sum_{j = 0}^{\lfloor q/2 \rfloor} |a_j(r*f-\mathbf{1})| \\ 
&= \sum_{j = 0}^{\lfloor q/2 \rfloor} |a_j (r*f)-a_j(\mathbf{1})| = \sum_{j = 1}^{\lfloor q/2 \rfloor} |a_j(r*f)| \\
&= \sum_{j = 1}^{\lfloor q/2 \rfloor} |\g_j a_j(r) a_j(f)|
= \sum_{j = 1}^{\lfloor q/2 \rfloor} |\l_j a_j(f)| \\ 
&\leq \l_1 \sum_{j = 1}^{\lfloor q/2 \rfloor} | a_j(f)| = \l_1 ||f-\mathbf{1}||_A.
\end{split}\end{equation*}
\end{enumerate}
Therefore all the conditions for Theorem 3.1 from \cite{RPT} are met. Hence we can already conclude that in the generalized clock model $\l_1 \br(T) < 1$ implies the absence of a RPT.\\

\noindent In the following we check the conditions for Theorem 3.2 of \cite{RPT}:
\begin{enumerate}[(i)]
\item First we like to note that for all $f\in P_+(\O_0/0)$
$$f(0) = \max_{k \in \O_0} f(k) \quad \mbox{ and } \quad f\Big(\ceil*{\frac{q+1}{2}}\Big)= \min_{k\in \O_0} f(k).$$
\item 
If $f$ is non-increasing, it is easy to see that $Mf=r*f$ is also non-increasing. Therefore, as a consequence of the recursion \eqref{rec}, all $f\in P_+(\O_0/0)$ are non-increasing.\\
Define the functional $\mathcal{L}: P(\O_0/0)\rightarrow\mathbb{R}$ that captures the weighted absolute value of the first Fourier coefficient by 
$$\mathcal{L}(g) := \sum_{k=0}^{q-1} g(k) \phi_1(k) = z_1 a_1(g).$$
Let $f\in P_+(\O_0/0).$ Then
\begin{equation*}\begin{split}
\mathcal{L}(Mf- \mathbf{1}) &= z_1 a_1(Mf- \mathbf{1} ) = z_1 a_1 (r*f- \mathbf{1})\\ 
&= z_1 a_1(r*f) = z_1^2 a_1(r )a_1(f) \\ 
&= \Big(\sum_{k=0}^{q-1} r(k)\phi_1(k)\Big) \Big(\sum_{l=0}^{q-1} f(l) \phi_1(l) \Big) \\
&= \Big(\sum_{k=0}^{q-1} r(k)\phi_1(k)\Big) \Big(\sum_{l=0}^{q-1} (f(l)- \mathbf{1} ) \phi_1(l) \Big) = \l_1 \mathcal{L}(f- \mathbf{1}). 
\end{split}\end{equation*}

\item Note that $\mathcal{L}(f)$ can be bounded from above by the $A$-norm of $f$:
$$|\mathcal{L}(f)| \leq z_1 ||f||_A$$
for all $f\in P_+(\O_0/0).$ 

\item 
Now, we want to find a constant $d>0$ s.t. $\mathcal{L}(f) \geq d ||f- \mathbf{1} ||_A$ for all $f\in P_+(\O_0/0)$. We already know that $|a_1(f)|=a_1(f)\geq |a_j(f)|$ for all $j\geq 2$. Therefore we can choose $d:= z_1/\floor*{\frac{q}{2}}$.
\end{enumerate}

Therefore all the prerequisites for Theorem 3.2 in \cite{RPT} are met as well. Hence we get that in the general clock model $\l_1 \br T> 1$ implies the existence of a RPT and we have finished the proof of Proposition \ref{pro:rpt}. \hfill $\Box$ \\

\begin{rem}
Let us once more consider the $q$-state generalized clock model with $q\in \{ 4,5 \}$ on the binary tree. We have already seen in Remark \ref{rem:lin} that Lemma \ref{lin} also holds for the case of transition matrices with negative eigenvalues as long as they are non-increasing. One can easily check that all the other prerequisites for Theorem 3.1 from \cite{RPT} are also met in this case and therefore $\l_1 \br T < 1$ still implies the absence of robust phase transitions. 

For the verification of the conditions of Theorem 3.2 from \cite{RPT} note that the space $ P_+(\O_0/0)$ is three-dimensional and $|a_1(f)| = a_1(f) \geq |a_2(f)|$ for every $f\in P_+(\O_0/0)$. It is then easily seen that for suitably chosen constant $c_4$ indeed $L(f) \geq c_4 ||f-1||_A$ for all $f\in P_+(\O_0/0)$. For $q=4,5$ the constant can be chosen to be $c_4 = 2 / z_1 $. All the other conditions of Theorem 3.2 are clearly met as well.

We conclude that for $q=4,5$ the sharp threshold for the existence of a RPT holds even for transition matrices with negative eigenvalues as long as the transfer matrix $M$ is non-increasing.
\end{rem}

\section{Non-robust phase transitions in the generalized clock model on the binary tree}

\subsection{Case q=5}
In this section we will take a closer look at the generalized clock model with local state space $\O_0=\lbrace 0,...,4 \rbrace$. Furthermore we will assume that $T$ is the binary tree, i.e., the Cayley tree $\CC\TT(2)$. The first row of the transfer matrix $M$ is given by 
$$r= (r_0,r_1,r_2,r_3,r_4) = (a,b,c,c,b),$$
where $a+2b+2c=1$. 
As $M$ is symmetric there are five eigenvalues $\l_i, i=0,...,4$, that are given by 
$$\l_m = \sum_{k=0}^{4}r_k \exp \big(-2\pi imk/5 \big).$$
This gives us $$\l_0 = 1, \quad \l_1 = a+2b\cos(2\pi/5)+2c\cos(4\pi/5) = \l_4$$ and $$\l_2 =a+2b\cos(4\pi/5)+2c\cos(2\pi/5)=\l_3.$$
In the following we will assume that $\l_1\leq \frac{1}{2}$ and $1=\l_0 > \l_1 \geq  \l_2$ are given in such a way that $M$ is non-increasing. The assumption $\l_1\leq\frac{1}{2}$ is made as this is the known critical value for the existence of a robust phase transition in the generalized clock model on the binary tree. The inequality $\l_1 \geq \l_2$ is necessary for $M$ being non-increasing. By the Fourier inversion formula the eigenvalues determine the first row of $M$. We have
$$r_l = \frac{1}{5} \sum_{m=0}^4 \l_m \exp\big(2\pi ilm/5\big)$$
and therefore
\begin{equation*}\begin{split}
r_0 &= \frac{1}{5} (1+2\l_1+2\l_2), \\
r_1 &= \frac{1}{5} (1+2\l_1 \cos(2\pi/5)+ 2\l_2 \cos(4\pi/5)) =r_4,\\
r_2 &= \frac{1}{5} (1+2\l_1 \cos(4\pi/5)+ 2\l_2 \cos(2\pi/5))=r_3.
\end{split}\end{equation*}


\noindent The recursion for the marginal distributions of the Gibbs measure can now be seen as acting on the orthogonal space of the constant function, that is the set of symmetric mass zero measures, $\{ p \in P(\O_0/0) \mid \sum_{k=0}^{q-1} p_k = 0 \}$. For $q=5$ this space is two-dimensional with basis $\lbrace \phi_1,\phi_2\rbrace$, where $$\phi_1:= c^{-1} (1,\cos(2\pi/5),\cos(4\pi/5),\cos(4\pi/5),\cos(2\pi/5))^T$$ 
and $$\phi_2:=c^{-1}(1,\cos(4\pi/5),\cos(2\pi/5),\cos(2\pi/5),\cos(4\pi/5))^T $$
with $$c:=\sqrt{1+2\cos^2(2\pi/5)+2\cos^2(4\pi/5)}$$ being a normalizing constant. 
For a probability distribution $\a \in \M_1(\O_0)$ given by
$$\a(j)= \frac{1}{q}+\sum_{k=1}^2 \a_k\phi_k(j),$$
with $\a_k$ being the Fourier coefficients, we have
$$M\a(j)= \frac{1}{q}+\sum_{k=1}^2 \l_k\a_k\phi_k(j).$$
Let $F_5$ be the non-linear map which represents the recursion. As we are operating on the binary tree, recursion equation \eqref{rec} boils down to
$$\a(j) = F_5 \a(j) := \frac{\left( \frac{1}{q} + \sum_{k=1}^2 \l_k\a_k\phi_k(j)\right)^2}{\sum_{j=0}^{4}\left( \frac{1}{q} + \sum_{k=1}^2 \l_k\a_k\phi_k(j)\right)^2}.$$
For the summands in the numerator we have
\begin{equation*}
(M \a(j))^2=\frac{1}{q^2}+\frac{2}{q}\sum_{k=1}^{2} \l_k \a_k \phi_k(j)
+\sum_{k=1}^{2} \sum_{l=1}^{2} \a_k \a_l \l_k  \l_l \phi_k(j)\phi_l(j).
\end{equation*}
By the addition theorem for the cosine follows
\begin{equation*}
\phi_k(j)\phi_l(j) = v(\phi_{k+l}(j) + \phi_{k-l}(j))
\end{equation*}
and
\begin{equation*}
\phi_k(j)^2 = v \phi_{2k}(j) + w,
\end{equation*}
where $v:= \frac{1}{2c}$ and $w:= \frac{1}{2c^2} =2v^2 = \frac{1}{q}$.
This gives us
\begin{equation*}\begin{split}
(M \a(j))^2 &=\frac{1}{q^2}+\frac{2}{q}\sum_{k=1}^{2} \l_k \a_k \phi_k(j) +v \sum_{1\leq k,l \leq 2, k\neq l} \a_k \a_l \l_k \l_l (\phi_{k+l}(j)+\phi_{k-l}(j)) \\
& \qquad +v\sum_{1\leq k\leq 2} \a^2_k \l_k^2 \phi_{2k}(j) 
+w\sum_{1\leq k\leq 2} \a^2_k \l^2_k
\cr
\end{split}
\end{equation*} 
and by orthogonality of the basis functions $\phi_1, \phi_2$ the normalization is just
\begin{equation*}\begin{split}
&\sum_{1\leq j\leq q } (M \a(j))^2=\frac{1}{q}+w q\sum_{1\leq k\leq 2} \a^2_k \l^2_k.
\end{split}
\end{equation*} 
Therefore the recursion can be written as
\begin{equation}\label{ref}\begin{split}
F_5 \a (j) &= \frac{1}{q}+\frac{1}{\frac{1}{q}+w q\sum_{1\leq k\leq 2} \a^2_k \l^2_k}
\Bigl(\frac{2}{q}\sum_{k=1}^{2} \l_k \a_k \phi_k(j) \\
+ & \quad v \sum_{1\leq k,l \leq 2, k\neq l} \a_k \a_l \l_k \l_l 
(\phi_{k+l}(j)+\phi_{k-l}(j))
+v\sum_{1\leq k\leq 2} \a^2_k \l_k^2 \phi_{2k}(j) 
\Bigr) 
\cr
&= \frac{1}{q} + \sum_{m=1}^2 \a'_m \phi_m(j).
\end{split}
\end{equation} 
For $q=5$ this gives
\begin{equation*}\begin{split}
&\frac{1}{\frac{1}{q}+w q\sum_{1\leq k\leq 2} \a^2_k \l^2_k}
\Bigl( \frac{2}{q}\sum_{k=1}^{2} \l_k \a_k \phi_k(j)\cr
&+\sum_{1\leq k,l \leq 2, k\neq l} \a_k \a_l \l_k \l_l 
v(\phi_{k+l}(j)+\phi_{k-l}(j))
+v\sum_{1\leq k\leq 2} \a^2_k \l_k^2 \phi_{2k}(j) 
\Bigr) 
\cr
&= \sum_{m=1}^2 \a'_m \phi_m(j).
\end{split}
\end{equation*}

Rewriting these last equations produces two fixed-point equations for the Fourier modes $\a_1$ and $\a_2$:
\begin{equation}\label{a1}
	\a_1 = \frac{1}{\frac{1}{5}+\a_1^2\l_1^2+\a_2^2\l_2^2} \left( \frac{2}{5}\l_1\a_1 + 2\l_1\l_2 v \a_1\a_2+ v\l_2^2 \a_2^2\right)
\end{equation}
and 
\begin{equation}\label{a2}
	\a_2 = \frac{1}{\frac{1}{5}+\a_1^2\l_1^2+\a_2^2\l_2^2} \left( \frac{2}{5}\l_2\a_2 + 2\l_1\l_2 v \a_1\a_2+ v\l_1^2 \a_1^2\right).
\end{equation}
We set $\l_1=\frac{1}{2}$ as this is the critical value for the existence of RPT on the binary tree. 
Note that $\a_1=0$ if and only if $\a_2=0$. We define 
$$A:= 20 \l_2^2\a_2^2.$$
From the first of these equations we get
\begin{equation*}\begin{split}
\a_1(4+5\a_1^2+A) &= 4\a_1+20\l_2 v \a_1 \a_2 + vA 
\end{split}\end{equation*}
and hence 
\begin{equation}\label{A}
A = \frac{20\l_2 v \a_1 \a_2- 5\a_1^3}{\a_1-v}
\end{equation}
for $\a_1 \neq v$.
In the case where $\a_1 = v$ we find $\a_2 = \frac{v}{4\l_2}$. Plugging these values into the second fixed-point equation shows that this solution only appears for $\l_2 = (\frac{1}{5}v + \frac{1}{4}v^3 + \frac{1}{16} v^3)/(\frac{2}{5}v + 2v^3) \approx 0.385417$. 

Continuing with the second equation for $\a_1 \neq v$ we may use \eqref{A} to reduce the order of $\a_2$ repeatedly. This leads to
\begin{equation*}\begin{split}
& \a_2 = \frac{8\l_2 \a_2 + 20 \l_2 v \a_1 \a_2 + 5 v \a_1^2}{4+5\a_1^2+\frac{20 \l_2 v \a_1 \a_2 -5\a_1^3}{\a_1-v}} \\
\implies & \a_2 \Big( 4\a_1+5\a_1^3-4v-5v\a_1^2+20\l_2 v\a_1 \a_2 - 5\a_1^3 \Big) \\ 
& \qquad = (\a_1-v) (8\l_2 \a_2+20\l_2 v\a_1 \a_2 +5v\a_1^2)\\
\implies & 4\a_2 (\a_1-v)-5v\a_1^2 \a_2 + \frac{v\a_1}{\l_2} \Big( \frac{20\l_2 v\a_1 \a_2 - 5\a_1^3}{\a_1-v}\Big) \\
& \qquad = (\a_1-v) (8\l_2 \a_2+20\l_2 v \a_1 \a_2+5v\a_1^2)\\
\implies & \Big(4\l_2 (\a_1-v)^2-5\l_2 v\a_1^2(\a_1-v) + 20 \l_2 v^2\a_1^2 - 8 \l_2^2(\a_1-v)^2 \\
& \qquad-20\l_2^2 v \a_1 (\a_1-v)^2\Big)\a_2    = 5v\a_1^4 + 5v \l_2 \a_1^2(\a_1-v)^2.
\end{split}\end{equation*}
For $P_3(\a_1)\neq 0$ we have 
\begin{equation}\label{eq:y}\begin{split}
\a_2 = f(\a_1) = \frac{P_4(\a_1)}{P_3(\a_1)},
\end{split}\end{equation}
where 
\begin{equation*}\begin{split}
P_3(\a_1) &= 4\l_2 (\a_1-v)^2-5\l_2 v\a_1^2(\a_1-v) + 20 \l_2 v^2\a_1^2 \\
& \qquad  - 8 \l_2^2(\a_1-v)^2 -20\l_2^2 v \a_1 (\a_1-v)^2
\end{split}\end{equation*}
and
$$P_4(\a_1) = 5v\a_1^4 + 5v \l_2 \a_1^2(\a_1-v)^2.$$
If $P_3(\a_1)=0$ there is only the trivial solution.  
Substituting \eqref{eq:y} for the first fixed-point equation leaves us to solve
$$5\a_1^3P_3^2(\a_1) + 20 \l_2^2 P_4^2(\a_1) \a_1 - 20 \l_2 v \a_1 P_3(\a_1) P_4(\a_1) - 20 v \l_2^2 P_4^2(\a_1) =0,$$
which is equivalent to 
\begin{equation} \begin{split}\label{quartic}
 \a_1^3 (\a_1 - v)^2  \Big[ &   \a_1^4 \left( \frac{125 \l_2^2}{2} + 200 \l_2^3 + 250 \l_2^4 \right)  \\
	& + \a_1^3 \left(-20 \sqrt{10} \lambda_2^2 - 75 \sqrt{10} \l_2^3 + 125  \sqrt{10} \l_2^4 \right)  \\
	& + \a_1^2 \left( 140 \l_2^2 - 345 \l_2^3 + 75 \l_2^4 \right)  \\
	& + \a_1 \left(-16 \sqrt{10} \l_2^2 + 64 \sqrt{10} \l_2^3 - 125 \sqrt{\frac{5}{2}} \l_2^4 \right) \\ 
	& + 8 \l_2^2 - 36 \l_2^3 + 40 \l_2^4  \Big] =:  \a_1^3 (\a_1 - v)^2 q_{\lambda_2}(\a_1). 
\end{split}\end{equation}
Let us write for the quartic function $q_{\l_2} (\a_1) = a\a_1^4 + b\a_1^3 + c\a_1^2 + d\a_1 + e$. 
Looking at the discriminant $\Delta$ of $q_{\l_2}$, i.e.,
\begin{equation*}\begin{split}
\D &=  256a^{3}e^{3}-192a^{2}bde^{2}-128a^{2}c^{2}e^{2}+144a^{2}cd^{2}e-27a^{2}d^{4}\\
&  \quad +144ab^{2}ce^{2}-6ab^{2}d^{2}e-80abc^{2}de+18abcd^{3}+16ac^{4}e \\
&  \quad -4ac^{3}d^{2}-27b^{4}e^{2}+18b^{3}cde-4b^{3}d^{3}-4b^{2}c^{3}e+b^{2}c^{2}d^{2},
\end{split}\end{equation*}
it can be shown by numerical calculations that there exists no non-trivial solutions for $\a_1$ if $0 \leq\l_2 < \l_2^c$, where $\l_2^c$ is a real root of $\Delta$ with  
$\l_2^c \approx 0.370748$. Indeed, in this regime we have $\Delta(\l_2) >0$ and $P(\l_2) >0$, where
$$P := 8ac - 3b^2.$$
This is a well-known sufficient condition for the non-existence of any real roots of the quartic function $q_{\lambda_2}$. For the critical value $\l_2^c$ we have that $\Delta(\l_2^c) = 0, P(\l_2^c)>0$ and $D(\l_2^c) \neq 0$, where 
$$D := 64a^3e - 16 a^2c^2 + 16 ab^2c - 16a^2bd - 3b^4.$$
This means that for $\l_2 = \l_2^c$ a double real root for $q_{\lambda_2}$ appears. 
For $\l_2 \in (\l_2^c, \tilde \l_2^c)$ where $\Delta(\tilde \l_2^c) = 0$ with $\tilde \l_2^c \approx 0.494119$ we see by numerical computations that $\Delta (\l_2) < 0$ and hence there exist exactly two distinct real roots for $q_{\l_2}$. In the regime where $\l_2 > \tilde \l_2^c$ we have that $\Delta(\l_2) >0$, $P(\lambda_2) < 0$ and $D(\lambda_2) <0$. Hence there exist four distinct real roots in this case. 
In the situation where $\l_2 = \tilde \l_2^c$ we get that $\Delta(\l_2) = 0, P(\l_2) <0 $, $D(\l_2) <0 $ and $\Delta_0(\l_2) \neq 0$, where $\Delta_0 := c^2 - 3bd + 12ae$. This means that there exists one real double root and two simple real roots in this case \cite{Re22}. 

\begin{figure}
 \centering
\begin{subfigure}{1\linewidth}
    \centering
    \includegraphics[height=8cm,width=8cm]{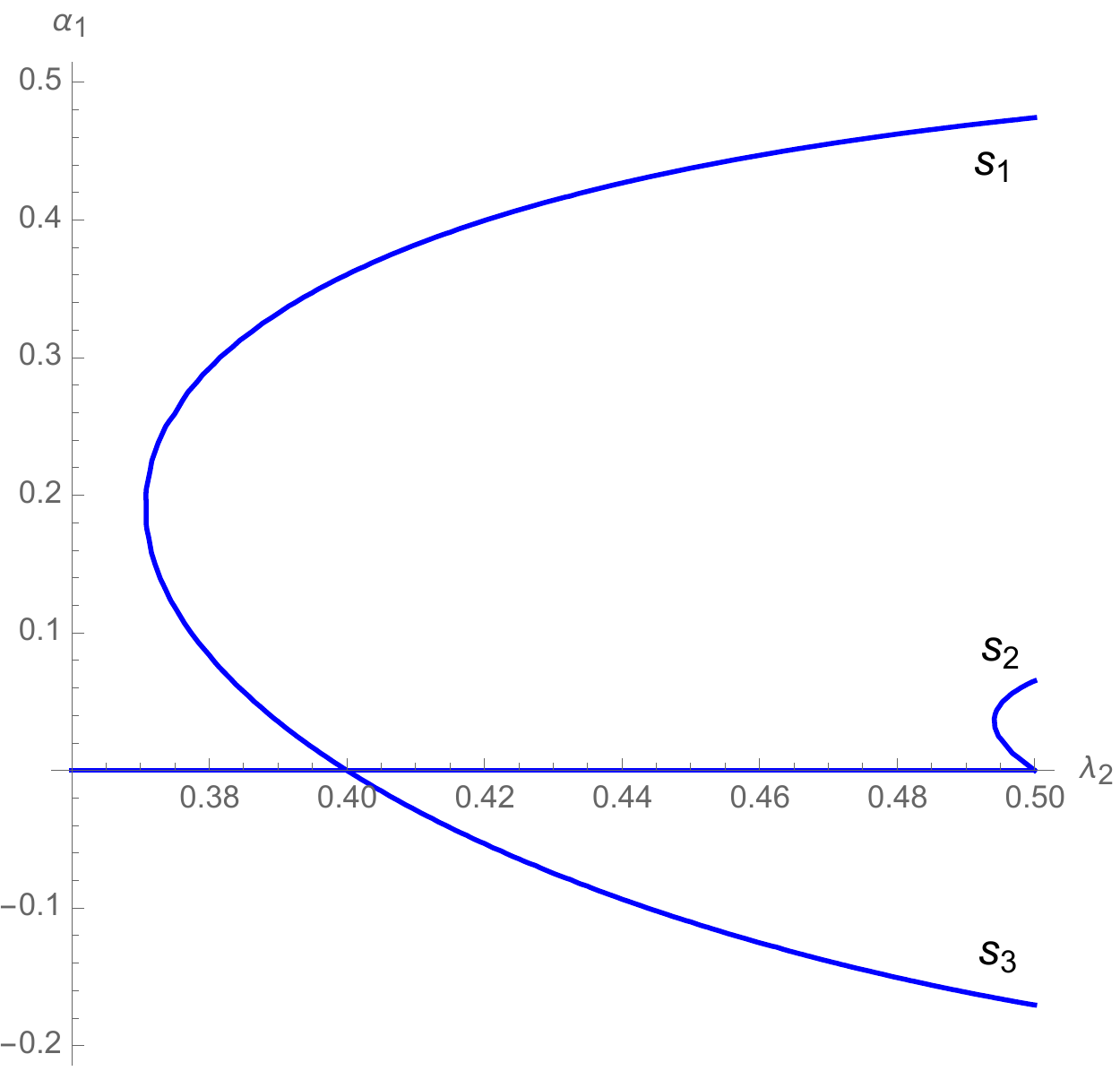} 
    \caption{The values of $\a_1$ on the vertical axis plotted against the corresponding values of the third largest eigenvalue $\l_2$ on the horizontal axis.}
    \subcaption*{}
    \label{fig:1}
 \end{subfigure}
 
\begin{subfigure}{1\linewidth}
   \centering
    \includegraphics[height=8cm,width=8cm]{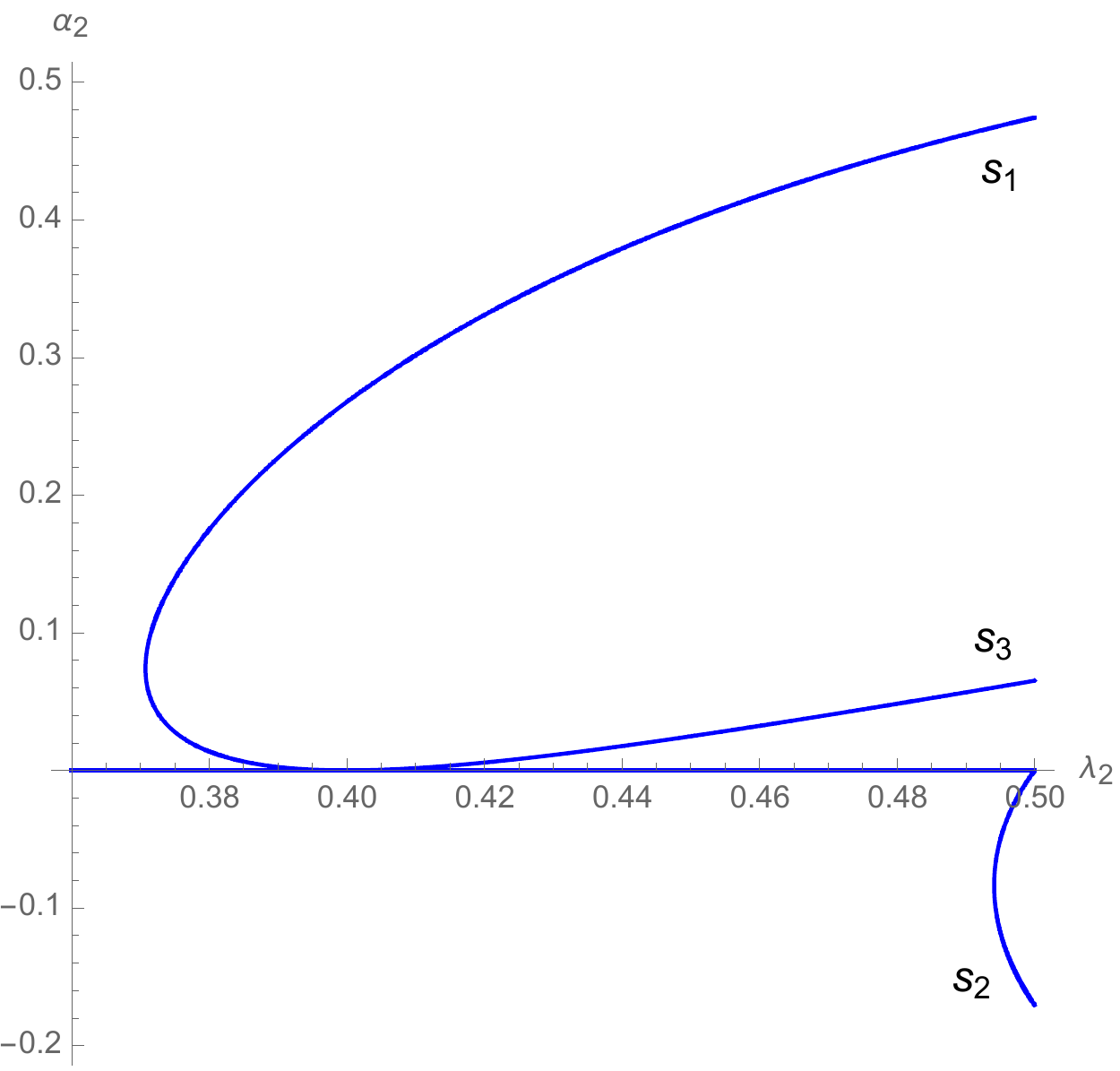} 
    \caption{The values of $\a_2$ on the vertical axis plotted against the corresponding values of the third largest eigenvalue $\l_2$ on the horizontal axis.}
    \label{fig:2}
 \end{subfigure}
\caption{The solutions to the boundary law equations \eqref{a1}, \eqref{a2} for $q=5$ in terms of the Fourier modes $\a_{1,2}$ depending on the third largest eigenvalue $\l_2$ of the transfer matrix $M$.}
\end{figure}

%


\begin{figure}
\centering
\begin{subfigure}{.35\textwidth}
  \centering
  \includegraphics[width=.3\linewidth]{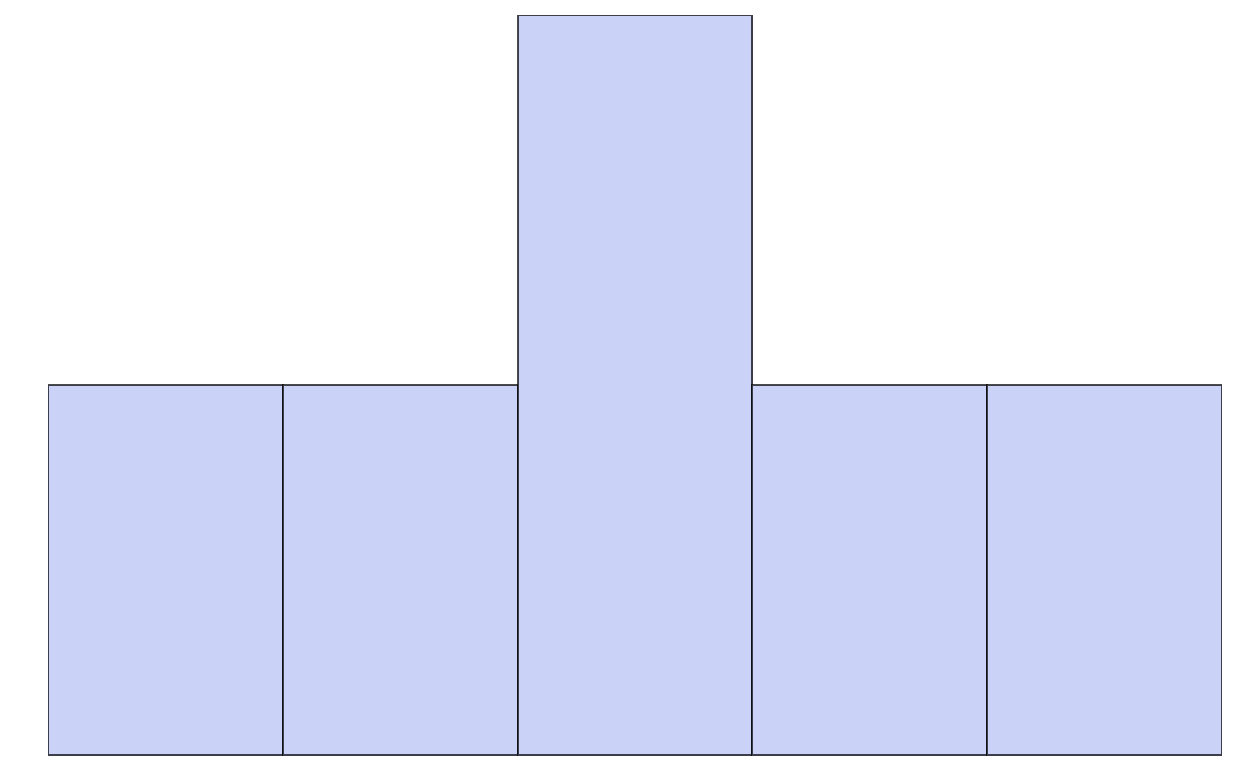}
  \label{fig:sub1}
\end{subfigure}%
\begin{subfigure}{.3\textwidth}
  \centering
  \includegraphics[width=.35\linewidth]{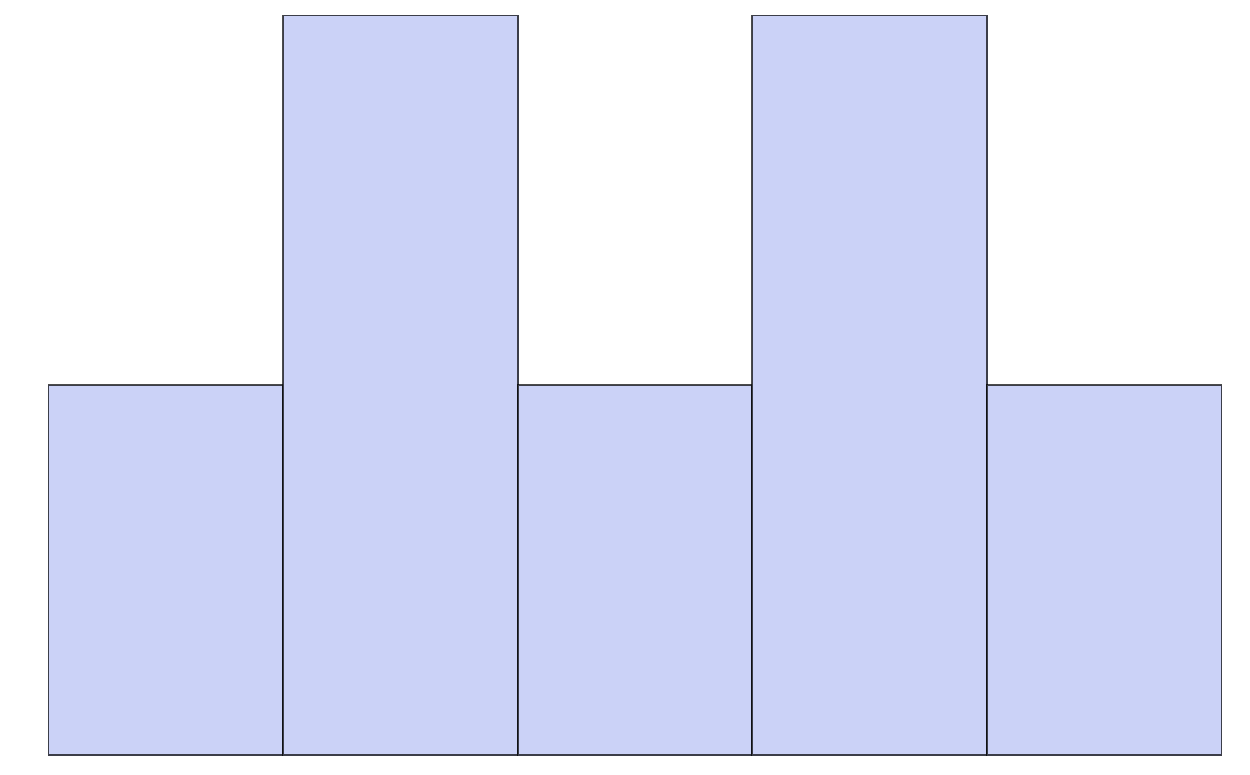}
  \label{fig:sub2}
\end{subfigure}
\begin{subfigure}{.3\textwidth}
  \centering
  \includegraphics[width=.35\linewidth]{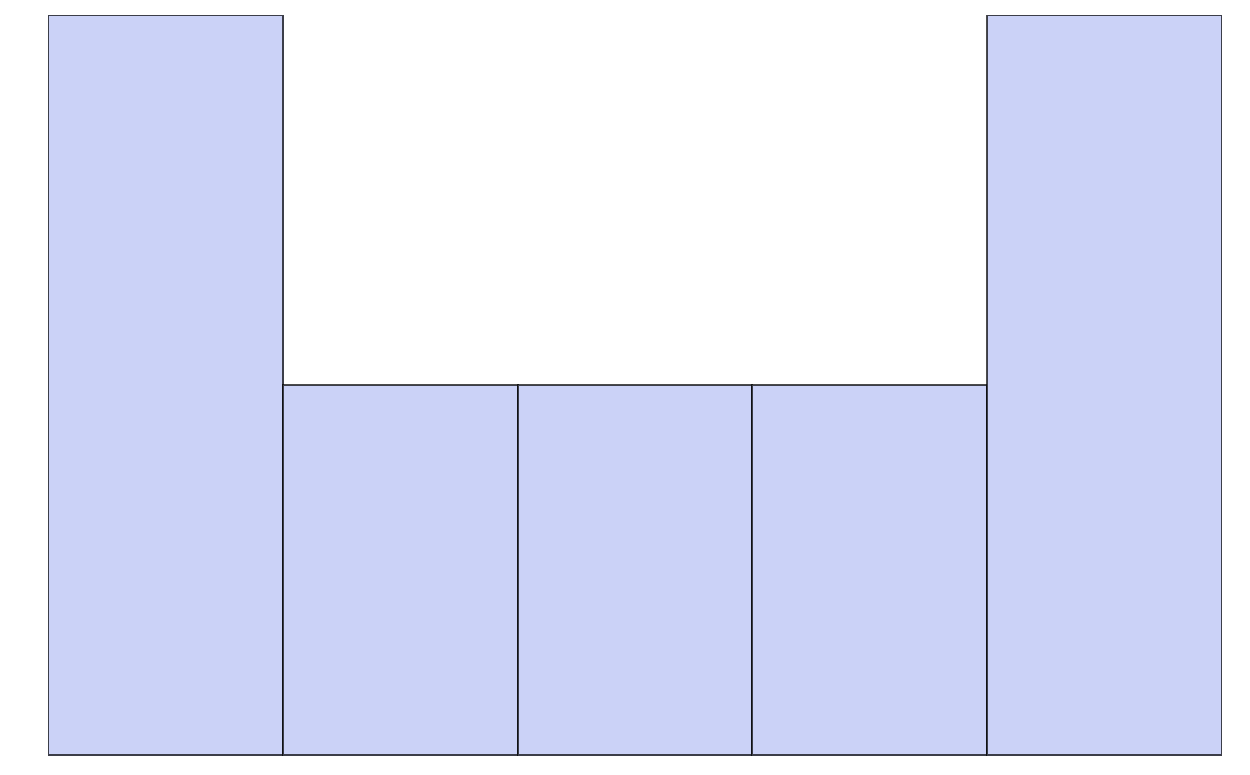}
  \label{fig:sub3}
\end{subfigure}
\caption{Symmetry of the solutions $s_1,s_2,s_3$ in the Potts model.}
\label{fig:bar}
\end{figure}


All the solutions to the polynomial equation \eqref{quartic} can be calculated by numerical techniques. For a plot of the derived coefficients $\a_1$ and $\a_2$ respectively against the corresponding value of $\l_2$, see Figures \ref{fig:1} and \ref{fig:2}. It can be checked by simple numerical calculations that all of these solutions correspond to a probability distribution. 

In accordance with our previous investigation we find that there only exists the trivial solution to the fixed-point equations if $\l_2 < \l_2^c$. 
At the critical value $\l_2^c$ an additional solution away from the trivial one appears, which then undergoes a pitchfork bifurcation. For $\l_2 = 0.4$ the lower branch then becomes $0$, which happens when  
$$e = 8 \l_2^2 - 36 \l_2^3 + 40 \l_2^4 = 0.$$
As we have predicted, an additional solution appears when $\l_2 = \tilde \l_2^c$. This solution then also undergoes a bifurcation for increasing values of $\l_2$ and the lower branch becomes $0$ for $\l_2 = 0.5$, which is again where $e=0$.  
So we have exactly four different solutions to our two-dimensional fixed-point equation when $\l_2 = 0.5$ which is of course when our generalized clock model turns into the Potts model with no external magnetic field. 
In this situation there are three additional solutions $s_1,s_2,s_3$ to the trivial one. The symmetries of the corresponding boundary laws are visualized in Figure \ref{fig:bar}. 

In the Potts model the transition matrix can be written in the form
\begin{equation*}
M_{i,j} = \begin{cases} e^{\b} /Z, &\mbox{if } i=j  \\ 
1/Z, & \mbox{if } i \neq j \end{cases}
\end{equation*}
with $\b>0$ and $Z=e^\b+q-1$ being a normalizing constant. Then $\l_1 = \l_2 = \frac{e^{\b}-1}{e^{\b}+q-1}$. Note that $2 \l_1=1$ if and only if $\theta := e^{\b} = q+1$. 

Let us comment on this value. Recall that in the $5$-state Potts model the translation-invariant states are described by the translation-invariant solutions for the boundary law equation \cite{KuRo14}: There is always 
the free solution (corresponding to open boundary conditions), and each boundary law has the property that its $5$ components 
take only two values. There are two different one-dimensional fixed-point equations for the case of a decomposition into either $1+4$ or $2+3$ components. 
Each of these fixed-point equations have different critical inverse temperatures at which solutions appear and bifurcate into an upper and a lower branch. 
It is a property of the recursion that on the binary tree at $\theta := e^{\b} = q+1$ the lower branches of these recursions become equal to the free solution. This behavior is similar for all Potts models on Cayley trees of any degree and the corresponding value $\theta_{cr}$ is known \cite{KuRo14}. 
As we will see shortly in Proposition \ref{pro:pottsRPT} this critical value is identical with the critical value for the existence of RPT for Cayley trees of any degree.

This is remarkable. Robust phase transition is caused by local instability of the equidistribution, that is small perturbations won't be attracted 
by the equidistribution but will converge to (one of) the stable symmetry breaking fixed points. 
On the other hand, the degeneration of the lower branches of the $m$-indexed boundary laws 
(with $m$ singled out components) is indicated by the same equation. 
Let $\theta_{RPT}$ denote the critical value of $\theta$ for the existence of RPT in the Potts model. 

\begin{pro}\label{pro:pottsRPT}
Let $T$ be a Cayley tree of degree $d \in \N$. In the Potts model with $q \in \N$ states we have $$\theta_{RPT}=\theta_{cr}.$$  
\end{pro}

\textit{Proof.} It is known that $\theta_{cr} = \frac{d+q-1}{d-1}$ \cite{KuRo14}. As $\theta = \theta_{RPT}$ iff $$\l_1= \frac{\theta-1}{\theta+q-1} = \frac{1}{d}$$
the claim is obviously true. \hfill $\Box$

\subsection{Case q=4} 
We still consider the generalized clock model on the binary tree $T$. Let the local state space be given by $\O_0 = \{0,...,3\}$. Again we assume $M$ to be non-increasing. The entries of the first row of the transfer matrix are given by
$$r_0 = \frac{1}{4}(1+2\l_1+\l_2),$$
$$r_1 = \frac{1}{4}(1-\l_2) = r_3,$$
$$r_2 = \frac{1}{4}(1-2\l_1+\l_2),$$
where $\l_i, i=0,1,2$ denote the eigenvalues of $M$. 
For the basis of the mass-zero measures we choose
$$\varphi_1 := c_1^{-1} (1, \cos \frac{\pi}{2}, \cos \pi, \cos \frac{3\pi}{2} ) = c_1^{-1} (1,0,-1,0), \quad c_1^{-1}:=\sqrt{2},$$
$$\varphi_2 := c_2^{-1} (1, \cos \pi, \cos 2\pi, \cos 3\pi ) = c_2^{-1} (1,-1,1,-1), \quad c_2^{-1}:=2.$$
The symmetric probability distributions can be represented as
$$\a(j) = \frac{1}{4}+\a_1 \varphi_1(j)+ \a_2\varphi_2(j).$$
The recursion $F_4$ of the distributions on the binary tree $T$ is then given by
$$\a(j) = F_4\a(j) := \frac{\left( \frac{1}{4} + \sum_{k=1}^2 \l_k\a_k\varphi_k(j)\right)^2}{\sum_{j=0}^{3}\left( \frac{1}{4} + \sum_{k=1}^2 \l_k\a_k\varphi_k(j)\right)^2}.$$
The summands in the numerator are given by
\begin{equation*}\begin{split}
(M\a(j))^2 = & \frac{1}{4^2} + \frac{1}{4}\Big(\a_1^2\l_1^2+\a_2^2\l_2^2\Big) + \Big(\frac{1}{2}\l_1\a_1+\a_1\a_2\l_1\l_2\Big)\varphi_1(j) \\ 
& + \Big(\frac{1}{2}\l_2\a_2+\frac{1}{2}\a_1^2\l_1^2\Big) \varphi_2(j).
\end{split}\end{equation*}
By orthogonality the normalization is just
$$\sum_{j=0}^3 (M\a(j))^2  =  \frac{1}{4} + \a_1^2\l_1^2 + \a_2^2\l_2^2$$
and hence
$$F_4 \a(j) = \frac{1}{4} + \frac{\frac{1}{2}\l_1\a_1+\a_1\a_2\l_1\l_2}{\frac{1}{4} + \a_1^2\l_1^2 + \a_2^2\l_2^2}\varphi_1(j) + \frac{\frac{1}{2}\l_2\a_2+\frac{1}{2}\a_1^2\l_1^2}{\frac{1}{4} + \a_1^2\l_1^2 + \a_2^2\l_2^2} \varphi_2(j).$$
This leads to the two fixed-point equations
\begin{equation}\begin{split}\label{q4}
\a_1 &= \frac{\frac{1}{2}\l_1\a_1+\a_1\a_2\l_1\l_2}{\frac{1}{4} + \a_1^2\l_1^2 + \a_2^2\l_2^2}, \\
\a_2 &= \frac{\frac{1}{2}\l_2\a_2+\frac{1}{2}\a_1^2\l_1^2}{\frac{1}{4} + \a_1^2\l_1^2 + \a_2^2\l_2^2}.
\end{split}\end{equation}
For the critical value $\l_1=1/2$ it follows from the first of these equations that $\a_1=0$ or 
\begin{equation}\begin{split}\label{eq:a1pm}
\a_1 = \pm \sqrt{2\l_2\a_2-4\l_2^2 \a_2^2}. 
\end{split}\end{equation}
In the case where $\a_1=\pm \sqrt{2\l_2\a_2-4\l_2^2 \a_2^2}$, it follows from the second  fixed-point equation that 
$$\a_2 = \frac{3\l_2 - 1}{2 (\l_2+\l_2^2)}.$$
Plugging this value for $\a_2$ into \eqref{eq:a1pm}, we get that $ \pm \sqrt{2\l_2\a_2-4\l_2^2 \a_2^2}$ only takes real values if and only if $\l_2 \in (1/3,1)$. Hence the critical value for $\l_2$ for the existence of a phase transition is $\l_2^c:=\frac{1}{3}$. An easy calculation shows that all the solutions to \eqref{q4} correspond to a probability distribution. We have thereby shown the following theorem.

\begin{thm}
Consider the $4$-state generalized clock model on the binary tree where the critical value for RPT is $\l_1= 1/2$. Then the spin model has Gibbs measures with single-site marginals different from the equidistribution if and only if $\l_2 \in (1/3,1/2)$.
\end{thm}

We have seen that there exist non-robust phase transitions in the generalized clock model on the binary tree for $q \in \{ 4,5 \}$. It is not difficult to prove that for a generalized clock model on a tree $T$ which exhibits no RPT there always exists another tree $T'$ s.t. $\br T = \br T' $ and $T'$ does not exhibit a PT. This is very much done by the same techniques proving the absence of RPT under the condition $\l_1 \br T <1$. The tree $T'$ is constructed by adding edges in series with sufficient length to the tree $T$, where the adding of these edges must be done sufficiently sparse such that $\br T = \br T' $. The recursion along these added series of edges then replaces the weakening of the coupling along the cutset in the notion of RPT (see the proof of Theorem 1.14 in \cite{RPT}). 

%
%


\begin{figure}
\centering
\begin{subfigure}{.5\textwidth}
  \centering
  \includegraphics[width=1.\linewidth]{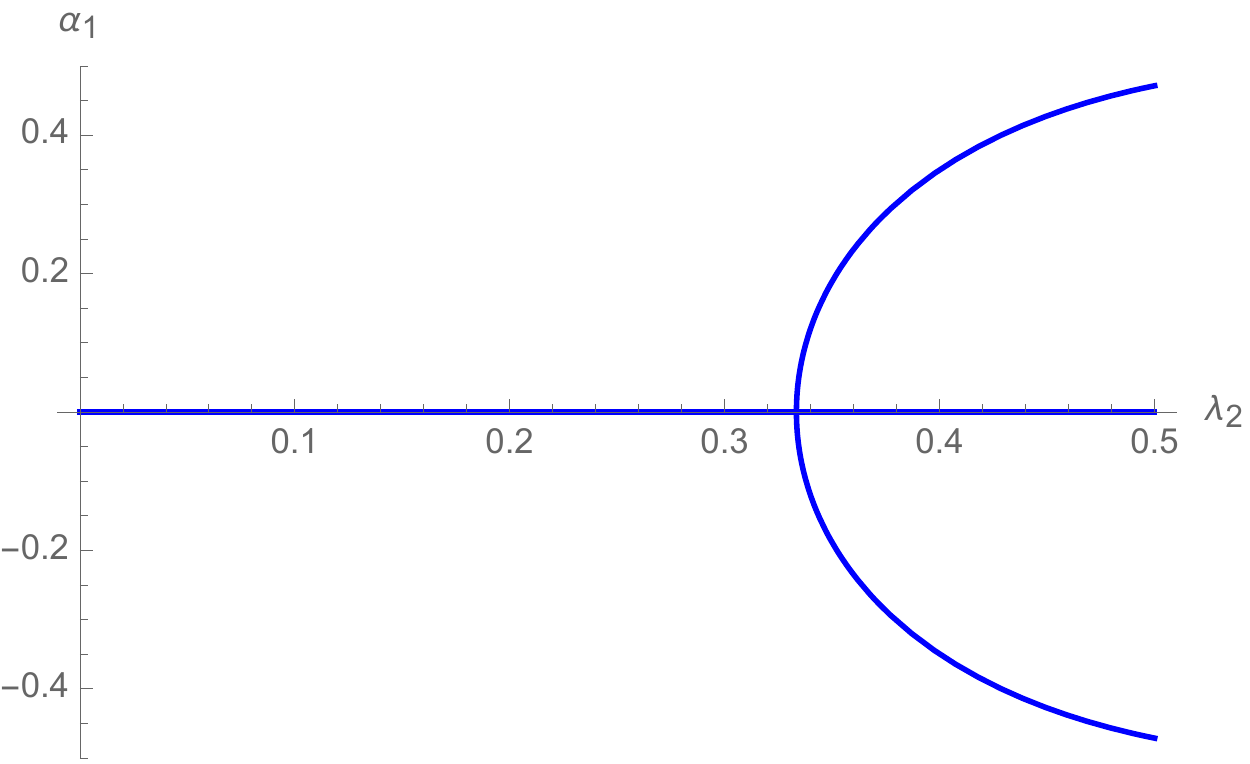}
  \label{fig:q4a1}
\end{subfigure}%
\begin{subfigure}{.5\textwidth}
  \centering
  \includegraphics[width=1.\linewidth]{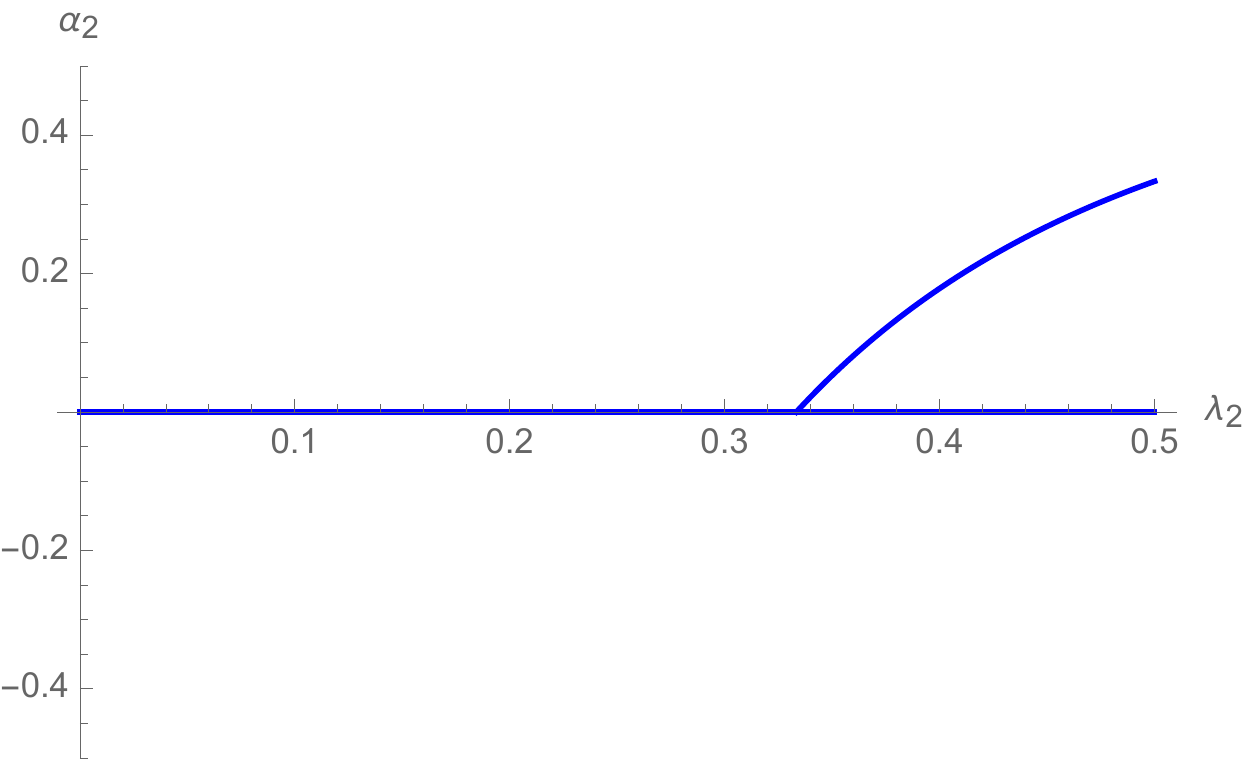}
  \label{fig:q4a2}
\end{subfigure}
\caption{The solutions to the boundary law equations \eqref{a1}, \eqref{a2} for $q=4$ in terms of the Fourier modes $\a_{1,2}$ depending on the third largest eigenvalue $\l_2$.}
\label{fig:q4}
\end{figure}

Let us map our solutions to the known solutions of the Potts model \cite{KuRo14}.
Note that at $\l_1=1/2$ there is a bifurcation of the $m=2$ branches from the free solution $\a =1$, 
where both branches exist only for bigger values of $\l_1$. Hence the only 
available solution is of type $m=1$, and here we only find the solution of the upper branch of the form, 
written in boundary law language, $(a,1,1,1), a \neq 1$. 
Also, our recursion acts 
in the subspace of these solutions which are symmetric w.r.t. $0$. 
Hence only two of the possible permutations persist, namely 
$(a,1,1,1), (1,1,a,1)$. These solutions are mapped to the upper and lower end of 
the pitchfork type branches in Figure \ref{fig:q4}. \\

In the Potts model the transfer matrix is given by $M(i,j) = \exp[\beta \mathbf{1}_{\{ i = j \}}]$ with inverse temperature $\beta>0$. It is known that the effective inverse temperature for the existence of a PT in the Potts model is given by $\beta_c = \frac{1}{2}\log(1+2\sqrt{q-1})$. As the second largest eigenvalue is given by $\l_2 = \frac{e^\beta - e^{-\beta}}{e^\beta + (q-1)e^{-\beta}}$ we get for the critical value $\l_2^c = \frac{2 \sqrt{q-1}}{q + 2 \sqrt{q-1}}$. For $q=4,5$ we have given the exact transition point for the existence of (non-robust) phase transitions at criticality, i.e. for $\l_1 = 1/2$. For the case $q=4$ we can even calculate the critical value $\l_2^c$ for the existence of non-robust phase transitions for any value of $\l_1 \in [0, 1/2]$: 
For $\a_1 \neq 0$ we get from the first equation of \eqref{q4} that
\begin{equation}\begin{split}
\a_1^2 \l_1^2 &= \frac{1}{2}\l_1 + \l_1 \l_2 \a_2 - \frac{1}{4}- \l_2^2\a_2^2 =: P_1(\a_2).
\end{split}\end{equation}
The second equation of \eqref{q4} gives us
\begin{equation}\begin{split}
\a_1^2 \l_1^2 &= - \l_2 \a_2 + \l_1 \a_2 + 2 \l_1\l_2 \a_2^2 =: P_2(\a_2). 
\end{split}\end{equation}
The critical line in $(\l_1,\l_2)$-space 
for the existence of multiple solutions to \eqref{q4} are given by the two equations $P_1(\a_2) = P_2(\a_2)$ and $P_1'(\a_2) = P_2'(\a_2)$. 
They are derived by the requirement that the parabolas described by $P_1$ and $P_2$ as a function of $\a_2$ have precisely one touching point described 
by a critical $\a_2$. It can be shown by elementary computations eliminating the $\a_2$ 
that the solutions to this system are given by $\l_1 = 4 \l_2 \frac{1-\l_2}{(1+\l_2)^2}$. The critical points are given by the dashed line in Figure \ref{PTvsRPT}.

\begin{figure}[ht]
    \centering
    \includegraphics[height=9cm,width=13cm]{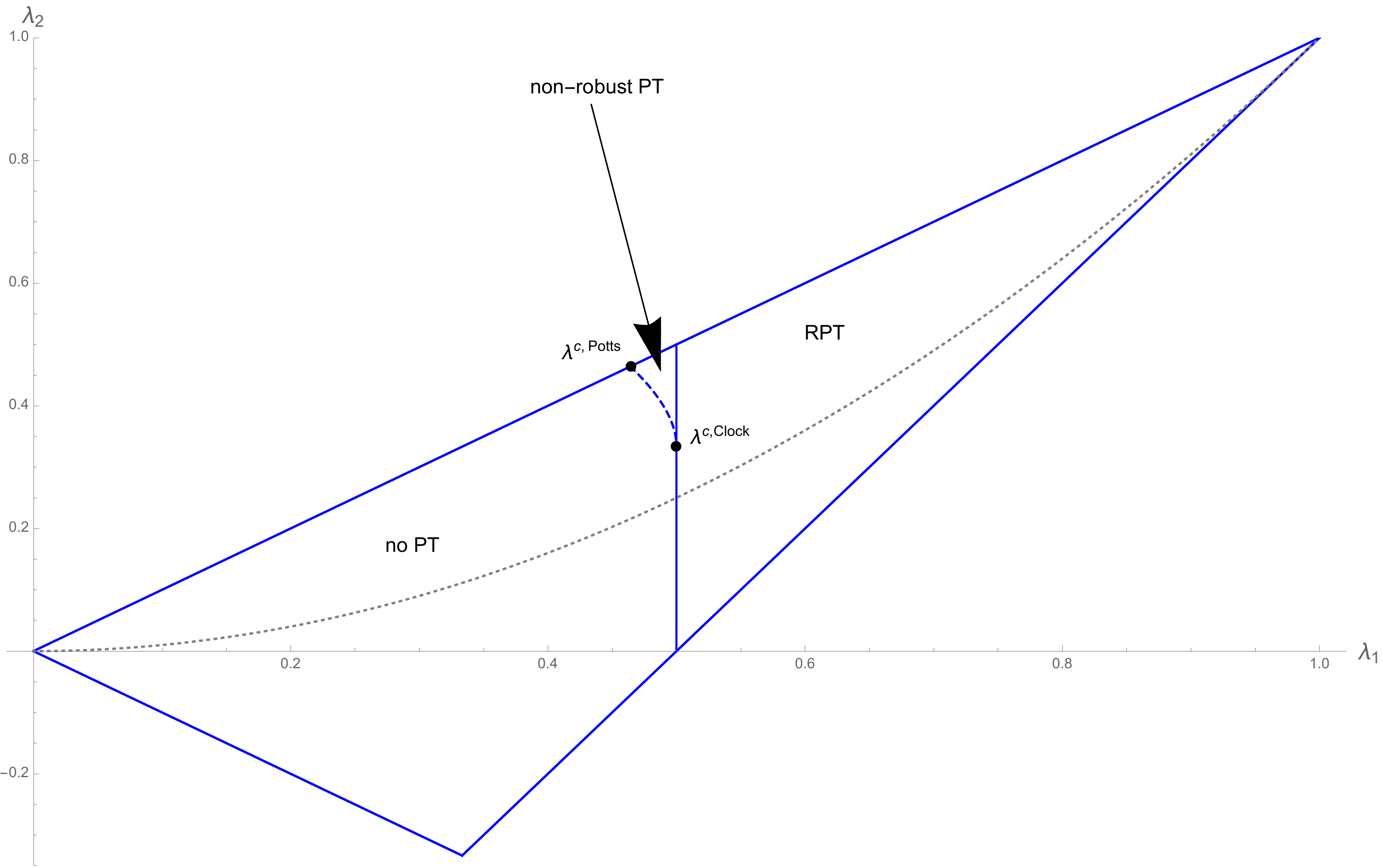} 
    \caption{
    The generalized clock model for $q=4$ can be described by the two eigenvalues $\l_{1,2}$ of the transfer matrix $M$. The region for the existence of non-robust phase transitions is given by the area above the dashed line inside the quadrangle. Note that the transition point in the Potts model is given by $\l_1= \l_2 = \frac{2\sqrt{q-1}}{q + 2 \sqrt{q-1}} \approx 0.4641$ and the transition point at the critical value for RPT in the generalized clock model on the binary tree is given by $\l_1 = 1/2$ and $\l_2 = 1/3$.
The eigenvalues of the transition matrix for the standard-clock model with scalarproduct-interaction are given by the dotted line. 
There are no non-robust PTs in the standard clock model.    
    }
    \label{PTvsRPT}
 \end{figure}

\begin{figure}[ht]
	\centering
	\includegraphics[height=9cm,width=13cm]{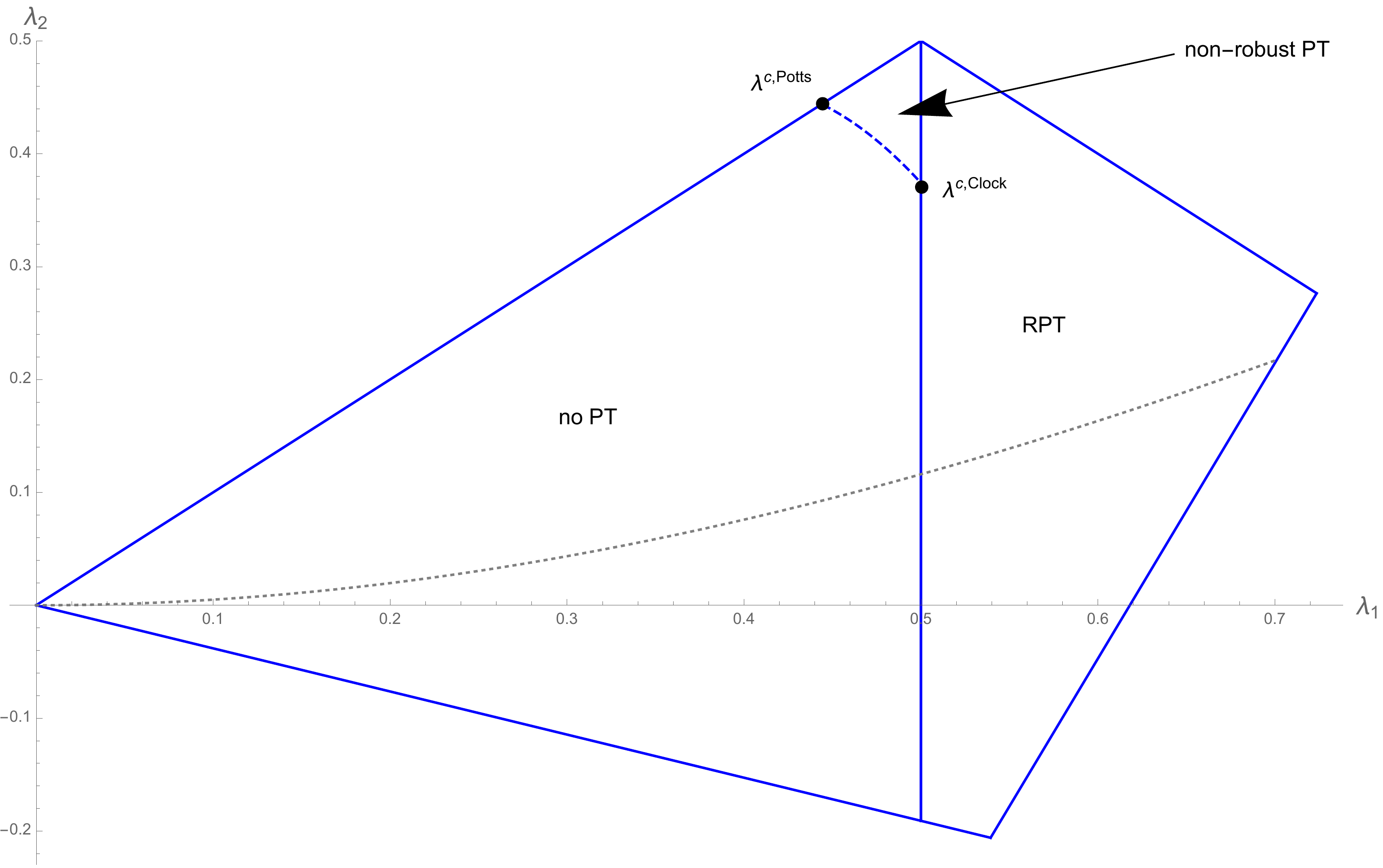} 
	\caption{The generalized clock model for $q=5$ can be described by the two eigenvalues $\l_{1,2}$ of the transfer matrix $M$. The region for the existence of non-robust phase transitions is given by the area above the dashed line inside the triangle. The eigenvalues for the standard-clock model are given by the dotted line. 
	}
	\label{PTvsRPTq5}
\end{figure} 

\begin{figure}[ht]
	\centering
	\includegraphics[height=4.5cm,width=7cm]{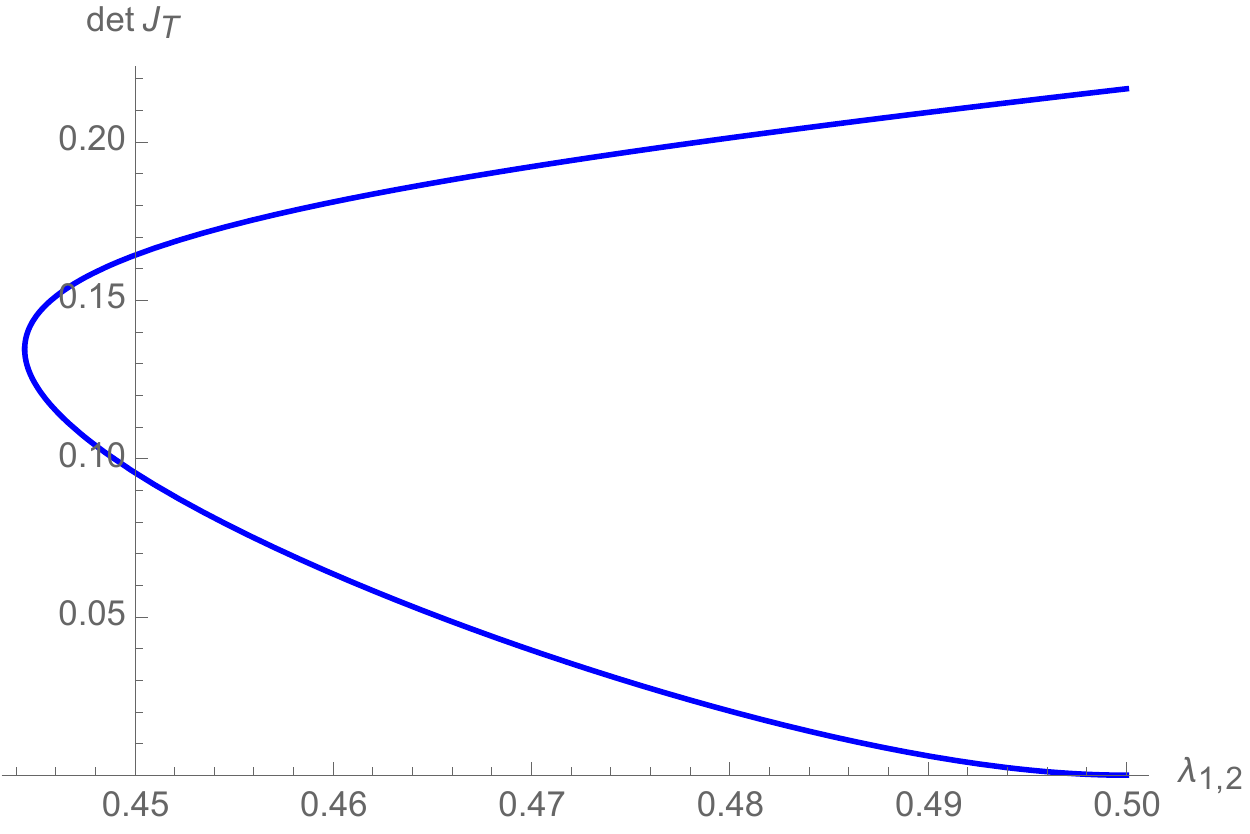} 
	\caption{In the Potts regime the values of the determinant of the Jacobi matrix of \eqref{fixed} for $\a_1 = \a_2 = \a_\pm$ are plotted against $\l = \l_1 = \l_2$.
	}
	\label{implicit}
\end{figure}

For $q=5$ the task to determine the transition line between the two regimes becomes more difficult. However, it is possible to use the implicit function theorem to prove the existence of non-robust phase transitions. 
In the Potts regime with $\l_1 = \l_2 = \l$ and $\l \in (4/9,1/2)$ there exist two b.l. solutions of the form $l=(a_\pm,1,1,1,1)$ with 
\begin{equation}\begin{split}
a_\pm = \left( \frac{1-e^\b}{2(q-1) } \pm \frac{\sqrt{5-4q - 2 e^\b + e^{2\b}}}{2(q-1)} \right)^{-1}
\end{split}\end{equation}
where $e^\b = \frac{1+ \l(q-1)}{1-\l}$. Here, $\l$ is the only non-trivial eigenvalue of the transfer matrix $M$. The Fourier modes $\a_1^\pm,\a_2^\pm$ of the solutions $a_\pm$ are then given by $\a_1^\pm = \a_2^\pm = \frac{2(1-a_\pm)}{5}$. Let $T : \R^2 \to \R^2$ be defined by 
\begin{equation}\label{fixed}
T(\a_1, \a_2) = \begin{pmatrix} \a_1 - \frac{1}{\frac{1}{5}+\a_1^2\l_1^2+\a_2^2\l_2^2} \left( \frac{2}{5}\l_1\a_1 + 2\l_1\l_2 v \a_1\a_2+ v\l_2^2 \a_2^2\right) \\
\a_2 - \frac{1}{\frac{1}{5}+\a_1^2\l_1^2+\a_2^2\l_2^2} \left( \frac{2}{5}\l_2\a_2 + 2\l_1\l_2 v \a_1\a_2+ v\l_1^2 \a_1^2\right)
\end{pmatrix}.
\end{equation}


\noindent
To apply the implicit function theorem we need to compute and discuss the Jacobian. We don't give details, but provide a plot 
of the Jacobian of $T$ for $\l_1 = \l_2$ and $\a_1 = \a_2 = \a_\pm$ which is given in Figure \ref{implicit}. Note that the Jacobi matrix becomes non-invertible right at the critical value for the existence of RPT, i.e., $\l_1 = \l_2 = 1/2$. This is the point where the lower branch of the b.l. solutions, i.e. $a_-$, turns into the free solution. For $\l_1 = \l_2 \in (4/9,1/2)$ the Jacobian is always invertible at the Potts solution. Hence for every $\l_1 =\l_2$ there exists an open neighborhood around that point in which there exist non-robust phase transitions.
Even though we did not give an analytic argument to show the form of the transition line between the regimes PT=RPT and PT$\neq$RPT, we can compute it numerically. The transition line for the case $q=5$ is given by the dashed line in Figure \ref{PTvsRPTq5}.

\section*{Acknowledgement}
This work is supported by Deutsche Forschungsgemeinschaft, RTG 2131
{\em High-dimensional Phenomena in Probability - Fluctuations and Discontinuity}. 
We thank Aernout van Enter for useful comments to an earlier draft of the manuscript.

\end{document}